\documentclass[12pt]{article}
\usepackage{amsthm}
\usepackage{varioref}
\usepackage{amsmath,amstext,amsthm,a4,amssymb,amscd}   
\usepackage[T1]{fontenc}
\usepackage[utf8]{inputenc}
\usepackage{lmodern}
\usepackage{array}
\usepackage{latexsym}
\usepackage{fancyhdr}
\usepackage{mathrsfs}\let\cal\mathscr
\usepackage[all]{xy}
\usepackage{makeidx}   
\usepackage{color}
\usepackage{pifont}
\usepackage{amsfonts,amssymb}
\usepackage{hyperref}
\usepackage{cleveref}
\usepackage[numbers,square]{natbib}
\usepackage[disable]{todonotes}
\usepackage{verbatim}
\usepackage{relsize}

\usepackage{avant}
\usepackage[T1]{fontenc}      

\usepackage{amsfonts,amssymb}  
 
\usepackage{graphicx}

\usepackage{natbib}
\newcommand \Om {\Omega}
\newcommand \om {\omega}
\newcommand \0 {\emptyset}

\renewcommand \i {\sqrt{-1}}
\renewcommand \leq {\leqslant}
\renewcommand \geq {\geqslant}
\newcommand{\longhookrightarrow}{\lhook\joinrel\longrightarrow}

\newcommand{\norm}[1]{\left\Vert #1\right\Vert}

\DeclareMathOperator{\Vol}{Vol}
\DeclareMathOperator{\End}{End}

\DeclareMathOperator{\Tr}{Tr}

\DeclareMathOperator{\diag}{diag}

\DeclareMathOperator{\Lie}{Lie}

\DeclareMathOperator{\Ad}{Ad}

\DeclareMathOperator{\Herm}{Herm}

\DeclareMathOperator{\sym}{sym}

\newcommand \dbar {\overline{\partial}}

\newcommand \< {\mathcal{h}}
\renewcommand \> {\mathcal{i}}

\newcommand \cinf {\CC^\infty}

\newcommand \Id {{\rm Id}}

\renewcommand \epsilon {\varepsilon}
\renewcommand \AA {{\cal A}}
\newcommand \CC {{\cal C}}

\newcommand \OO {{\mathcal O}}

\def\g{\mathfrak{g}}
\def\t{\mathfrak{t}}
\def\u{\mathfrak{u}}

\def\Im{{\rm Im}}

\newcommand \R {\mathbb R}

\newcommand \CP {\mathbb C\mathbb P}
\newcommand \C {\mathbb C}

\newcommand \N {\mathbb N}

\newcommand \Z {\mathbb Z}

\newcommand \IT {\mathbb T}

\newcommand \fl {\rightarrow}

\newcommand \ignore[1] {}

\theoremstyle{plain}
\newtheorem{theorem}{Theorem}[section]

\newtheorem{cor}[theorem]{Corollary}
\newtheorem{prop}[theorem]{Proposition}
\theoremstyle{definition}
\newtheorem*{ackn*}{Acknowledgements}
\newtheorem{defi}[theorem]{Definition}

\numberwithin{equation}{section}

\crefname{equation}{}{}
\crefname{lem}{Lemma}{Lemmas}
\crefname{theorem}{Theorem}{Theorems}
\crefname{cor}{Corollary}{Corollaries}
\crefname{ex}{Example}{Examples}
\crefname{defi}{Definition}{Definitions}
\crefname{prop}{Proposition}{Propositions}
\crefname{section}{Section}{Sections}
\crefname{subsection}{Section}{Sections}
\crefname{rmk}{Remark}{Remarks}
\crefname{nota}{Notation}{Notations}

\hypersetup{
    colorlinks=true,
    linkcolor=red,
    filecolor=magenta,      
    urlcolor=cyan,
}

\begin{document}

\title{\bf{Asymptotics of unitary matrix elements in canonical bases
}}
\author{Louis IOOS}
\date{}

\maketitle

\newcommand{\Addresses}{{
  \bigskip
  \footnotesize

  \textsc{Laboratoire Analyse, Géométrie et Modélisation, UMR 8088 CNRS,
CY Cergy-Paris Université, 95302 Cergy-Pontoise, France}\par\nopagebreak
  \textit{E-mail address}: \texttt{louis.ioos@cyu.fr}

}}

\begin{abstract}

We compute the asymptotics of matrix elements in the Gelfand-Zetlin bases of irreducible representations
of the unitary group as the highest weight goes to infinity,
in terms of the symplectic geometry of the associated coadjoint orbit.
This uses tools of Berezin-Toeplitz quantization, and
recovers as a special case the asymptotics of Wigner's
$d$-matrix elements for the spin representations in quantum mechanics.

\end{abstract}

%
%
%
%
%

\section{Introduction}

Let $n\in\N^*$, and consider the sequence of inclusions
\begin{equation}\label{GLinc}
U(1)\subset U(2)\subset\cdots\subset U(n)
\end{equation}
of the unitary groups,
where for any $k\in\N$ with $k\leq n$,
the subgroup $U(k)\subset U(n)$
acts on the canonical decomposition
$\C^n=\C^k\oplus\C^{n-k}$
in the usual way on the
first factor and as the identity on the second.
A classical result of Weyl, which we recall in \cref{Weylprop},
states that any unitary irreducible representation
$V$ of the unitary group $U(n)$ admits a unique decomposition
\begin{equation}
V|_{U(n-1)}=\bigoplus_{j\in J}\,V_j
\end{equation}
into distinct irreducible representations $V_j$ of $U(n-1)\subset U(n)$
by restriction,
so that each irreducible representation of $U(n-1)$ appears with
multiplicity at most $1$ in this decomposition. Applying the same result to
$V_j$ for each $j\in J$, this further implies that
$V_j$ decomposes as a direct sum of distinct irreducible representations
of $U(n-2)\subset U(n-1)$ as before.
By downward iteration along the
sequence of inclusions \cref{GLinc}, we thus obtain a canonical decomposition
\begin{equation}\label{GZdecflaintro}
V=\bigoplus_{\nu\in\Gamma}\,\C_\nu\,,
\end{equation}
where each component $\C_\nu\subset V$ is of complex dimension $1$
as an irreducible representation of $U(1)\subset U(n)$. This decomposition
was implicitly used by Gelfand and Zetlin in \cite{GZ50} to establish explicit
formulas for matrix elements of irreducible representations of $U(n)$,
and a basis $\{e_\nu\}_{\nu\in\Gamma}$ of $V$ compatible with the decomposition
\cref{GZdecflaintro} is called a \emph{Gelfand-Zetlin basis}.

The main results of this paper are \cref{mainth,toricth},
establishing the asymptotics of matrix elements in the Gelfand-Zetlin bases of
unitary irreducible representations of $U(n)$ as the \emph{highest weight} $\lambda=(\lambda_1,\lambda_2,\dots,\lambda_n)\in\Z^n$
goes to infinity. We compute these asymptotics in terms of
the symplectic geometry of the associated
\emph{coadjoint orbit} $X_\lambda\subset\u(n)^*$ for the
coadjoint action of $U(n)$ on the dual $\u(n)^*$ of its Lie algebra,
endowed with its natural symplectic form
$\om\in\Om^2(X_\lambda,\R)$ given in \cref{KKSprop}.
Assuming that $X_\lambda$ has maximal dimension
among coadjoint orbits in $\u(n)^*$,
Guillemin and Sternberg constructed in \cite{GS83} a remarkable
integrable system on $(X_\lambda,\om)$, called the \emph{Gelfand-Zetlin system}
and given by a continuous map
\begin{equation}\label{GZsystintro}
M:X_\lambda\longrightarrow\Delta_\lambda\subset\R^{\frac{n(n-1)}{2}}\,,
\end{equation}
which we describe in \cref{GZsystemprop},
whose image $\Delta_\lambda:=\Im(M)$ is a convex polytope,
called the \emph{Gelfand-Zetlin polytope}, defined by
\begin{multline}\label{GZpolytopeintro}
\Delta_\lambda=\{(\nu_j^{(k)})_{1\leq j\leq k\leq n-1}\in\R^{\frac{n(n-1)}{2}}~\big|
~\lambda_j\geq\nu_j^{(n-1)}\geq\lambda_{j+1}\,,\,
\nu_j^{(k)}\geq\nu_j^{(k-1)}\geq \nu_{j+1}^{(k)}\,\}\,.
\end{multline}
The map \cref{GZsystintro}
is smooth over the interior $\Delta_\lambda^0\subset\Delta_\lambda$, and induces
\emph{action-angle coordinates} over the open dense subset
$M^{-1}(\Delta_\lambda^0)\subset X_\lambda$. In particular, for any $v\in\Delta_\lambda^0$,
the fibre $M^{-1}(v)\subset X_\lambda$
is a Lagrangian torus inside $(X_\lambda,\om)$.
In \cite[\S\,6]{GS83}, Guillemin and Sternberg
showed that the components of the Gelfand-Zetlin decomposition
\cref{GZdecflaintro} of the irreducible representation
of $U(n)$ of highest weight $\lambda\in\Z^n$ are naturally
parametrized by the integral points
\begin{equation}\label{W=DZ}
\Gamma_\lambda:=\Delta_\lambda\cap\Z^{\frac{n(n-1)}{2}}
\end{equation}
inside the Gelfand-Zetlin polytope \cref{GZpolytopeintro}.

In the heuristics of \emph{Geometric Quantization}, the coadjoint orbit 
represents a phase space of classical mechanics, while the irreducible representation
represents the corresponding space of quantum states,
called the \emph{holomorphic quantization} of $(X_\lambda,\om)$.
The classical integrable system
\cref{GZsystintro} over $(X_\lambda,\om)$ then corresponds to the quantum integrable
system on $V$ given by the commutative algebra of diagonal operators
in the decomposition \cref{GZdecflaintro}, and the correspondence
\cref{W=DZ} can then be interpreted as the old quantum-classical
correspondence between integral orbits of a classical integrable system
and eigenstates of the corresponding quantum integrable system.
In the language of Guillemin and Sternberg in \cite{GS83}, this amounts to an
equivalence between holomorphic quantization and
\emph{Bohr-Sommerfeld quantization},
which are the geometric quantizations induced by two different \emph{polarizations},
the Kähler polarization on one hand and the real polarization
associated with the classical integrable system \eqref{GZsystintro} on the other.

The general idea behind the main \cref{mainth,toricth} of this paper
lies in this context of Geometric Quantization, and consists
in studying the so-called \emph{semiclassical limit}, which is the regime
where the scale becomes so large that one recovers the laws of classical
mechanics
from the laws of quantum mechanics. To describe this semiclassical limit,
we use another fundamental result of Weyl, which we recall in
\cref{highestwghtth},
giving a natural parametrization of the unitary irreducible representations
of $U(n)$ by the $n$-tuples of integers
$\lambda=(\lambda_1,\lambda_2,\dots,\lambda_n)\in\Z^n$ satisfying
$\lambda_1\geq\lambda_2\geq\cdots\geq\lambda_n$,
called the \emph{highest weight} of the corresponding representation
denoted by $V(\lambda)$.
On the other hand, given such a weight $\lambda\in\Z^n$, one can consider
the coadjoint orbit $X_\lambda\subset\u(n)^*$ passing through the diagonal
matrix $\diag(\lambda_1,\lambda_2,\dots,\lambda_n)\in\u(n)^*$
under the natural identification \cref{Herm=un}
of $\u(n)^*$ with the space of $n\times n$ Hermitian matrices.
Following for instance \cite[(8.13)]{GW09}, we
consider the \emph{$\rho$-shift} defined by
\begin{equation}\label{rhobar}
\rho:=(n-1,n-2,\dots,1,0)\in\Z^n\,.
\end{equation}
As one can see directly from \eqref{GZpolytopeintro}
and \eqref{W=DZ},
the set $\frac{1}{p}\Gamma_{p\lambda-\rho}$
becomes dense in $\Delta_\lambda\subset\R^{\frac{n(n-1)}{2}}$ as $p\in\N$ tends to infinity.
This type of behavior as $p\to+\infty$ is characteristic of a semiclassical limit
in Geometric Quantization.
Such parametrization of irreducible representations of $U(n)$
where one includes a $\rho$-shift \eqref{rhobar}
is sometimes called
the \emph{Harish-Chandra parametrization}.


Let us assume that the weight $\lambda\in\Z^n$ is
\emph{regular}, so that $\lambda_1>\lambda_2>\cdots>\lambda_n$,
and for any $v\in\Delta_\lambda$,
write
\begin{equation}
\Lambda_v:=M^{-1}(v)\subset X
\end{equation}
for the associated Lagrangian fibre of the Gelfand-Zetlin system
\cref{GZsystintro} over the Gelfand-Zetlin polytope
\eqref{GZpolytopeintro}.
Let $(v_p,\,w_p\in \frac{1}{p}\Gamma_{p\lambda-\rho})_{p\in\N}$ be two sequences
respectively 
converging to some $v,\,w\in\Delta_\lambda^0\subset\Delta_\lambda$.
Choose a base point $x_p\in g\Lambda_{v_p}\cap\Lambda_{w_p}$
for each $p\in\N$, and for any $x\in g\Lambda_{v_p}\cap\Lambda_{w_p}$,
we write
$\eta(x)\in\R$ for the symplectic area of a disk whose
boundary is given by a path
in $g\Lambda_{v_p}$ going from
$x_p\in g\Lambda_{v_p}\cap\Lambda_{w_p}$ to
any $x\in g\Lambda_{v_p}\cap\Lambda_{w_p}$
followed by a path in $\Lambda_{w_p}$ returning to $x_p\in g\Lambda_{v_p}\cap\Lambda_{w_p}$.
For each
$1\leq j\leq \frac{n(n-1)}{2}$,
write $M_j\in\cinf(X,\R)$ for the $j^{\text{th}}$ component of the map
\eqref{GZsystintro} in $\R^{\frac{n(n-1)}{2}}$, and
write $\{\cdot,\cdot\}$ for the Poisson bracket on
$\cinf(X_\lambda,\R)$ induced by the symplectic form of $(X_\lambda,\om)$.
The first main result of this paper is the following.

\begin{theorem}\label{mainth}
For any $p\in\N$, let
$\{e_\nu\}_{\nu\in \Gamma_{p\lambda-\rho}}$ be a Gelfand-Zetlin basis of the unitary
irreducible
representation of $U(n)$ with highest weight $p\lambda-\rho$.

Then for any $g\in U(n)$ such that
$g\Lambda_{v}\cap\Lambda_{w}=\0$ and for any $k\in\N$,
there exists $C_k>0$ such that for all $p\in\N$, we have
\begin{equation}\label{mainflaexp}
\left|\<ge_{pv_p},e_{pw_p}\>\right|\leq C_k\,p^{-k}\,.
\end{equation}

Furthermore, for any $g\in U(n)$ such that
the intersection
$g\Lambda_{v}\cap\Lambda_{w}$
is transverse,
one can choose the basis vectors so that
we have the following asymptotic expansion as $p\to+\infty$,
\begin{equation}\label{mainfla}
p^{\frac{n(n-1)}{4}}\<ge_{pv_p},e_{pw_p}\>=\sum_{x\in g\Lambda_{v_p}\cap\,\Lambda_{w_p}}
\frac{\sqrt{-1}^{\kappa(x)} e^{2\pi\sqrt{-1} p\eta(x)}}{
\big|\det\,(\{g_*M_j,M_k\}(x))_{j,k}\big|^{\frac{1}{2}}}+O(p^{-1})\,,
\end{equation}
with $\kappa(x_p)\in\Z/4\Z$ constant after some rank for any
converging
$(x_p\in g\Lambda_{v_p}\cap\Lambda_{w_p})_{p\in\N}$.

\end{theorem}

\cref{mainth} is proved in \cref{proofmainthsec} for
\emph{regular weights}
$\lambda\in\Z^n$, so that $\lambda_1>\lambda_2>\cdots>\lambda_n$,
which corresponds to the fact that the associated coadjoint orbit $X_\lambda\subset\u(n)^*$
is of maximal dimension among all coadjoint orbits
in $\u(n)^*$. In
\cite[p.\,122]{GS83} and \cite[p.\,229]{GS83b},
Guillemin and Sternberg indicate how to extend their construction of the
Gelfand-Zetlin system \cref{GZsystintro} to general coadjoint orbits,
but only the regular case is described precisely
enough in \cite{GS83} for our purposes.

However, in the most singular case of
$\lambda=(\lambda_1,0,\dots, 0)\in\Z^n$,
the geometric picture greatly simplifies, and the second main result of our paper is
an extension and refinement of \cref{mainth} to that case.
In fact, the Gelfand-Zetlin
decomposition \cref{GZdecflaintro} of the irreducible representation
of $U(n)$ with highest weight $(\lambda_1,0,\dots,0)$ coincides with its \emph{weight decomposition}
into irreducible representations of the $(n-1)$-dimensional
subtorus $T_0\subset U(n)$
consisting of diagonal matrices with highest-left coefficient equal to $1$.
On the other hand, the coadjoint action of $T_0\subset U(n)$ makes
$X_\lambda\subset\u(n)^*$ into a \emph{toric manifold}, and the
Gelfand-Zetlin system \cref{GZsystintro} coincides
with the associated \emph{moment map}
\begin{equation}\label{Mtoric}
\mu:X_\lambda\longrightarrow\Delta_\lambda\subset\R^{n-1}\,,
\end{equation}
so that the Gelfand-Zetlin polytope coincides with the
\emph{Delzant polytope} of $(X_\lambda,\om)$ as a toric manifold.
The correspondence \cref{W=DZ} then translates into the well-known
fact that the weights for the torus action of $T_0\subset U(n)$ on
the holomorphic quantization of a toric manifold are given by the integral points
inside its Delzant polytope. In the case at hand, the coadjoint orbit
$(X_\lambda,\om)$ naturally identifies as a toric manifold
with the projective space $\CP^{n-1}$ of complex
lines inside $\C^{n}$ endowed with its canonical symplectic form
of volume $\lambda_1\in\N$.

In the following theorem, we assume without loss of generality that
$\lambda=(1,0,\dots,0)$, so that as we explain in \cref{projsec},
for any $p\in\N$ we have
\begin{equation}\label{WplsubDlintro}
\frac{1}{p}\,\Gamma_{p}=\Delta\cap\Big(\frac{1}{p}\,\Z^{n-1}\Big)\,,
\end{equation}
where $\Gamma_p\subset\Z^n$ is the set of weights
for the action of $T_0\subset U(n)$
on the irreducible representation of $U(n)$
with highest weight $(p,0,\dots,0)\in\Z^n$
and $\Delta:=\Delta_{(1,0,\dots,0)}\subset\R^{n-1}$
is the Delzant polytope of the standard projective space
$\CP^{n-1}$.
For any $v\in\Delta$, we write
\begin{equation}
\Lambda_v:=\mu^{-1}(v)\subset\CP^{n-1}
\end{equation}
for the associated fibre of the corresponding moment map \cref{Mtoric}.
Let $(v_p,\,w_p\in \frac{1}{p}\,\Gamma_p)_{p\in\N}$
be two sequences belonging to the same face
of $\Delta$ after some rank and
respectively converging to some
$v,\,w\in\Delta$.
For any $g\in U(n)$ such that the intersection,
$g\Lambda_{v}\cap\Lambda_{w}$ is clean,
write $g\Lambda_{v}\cap\Lambda_{w}=\cup_{q=0}^m Y^{(q)}$
for the decomposition into connected components,
and  let
$\eta_{p}^{(q)}\in\R$ be the symplectic area of a
disk whose boundary is given by a path
in $g\Lambda_{v_p}$ going from
any point $x_p\in g\Lambda_{v_p}\cap\Lambda_{w_p}$
approaching $Y^{(0)}$ as $p\to+\infty$
to any point $y_p\in g\Lambda_{v_p}\cap\Lambda_{w_p}$
approaching $Y^{(q)}$ as $p\to+\infty$
followed by a path in $\Lambda_{w_p}$ returning to
$x_p\in g\Lambda_{v_p}\cap\Lambda_{w_p}$.
The second main result of this paper is the following.

%

\begin{theorem}\label{toricth}
For every $p\in\N$,
let $\{e_\nu\}_{\nu\in \Gamma_p}$ be a weight basis
for the torus action $T_0\subset U(n)$ on the irreducible
representation of $U(n)$ with highest weight $(p,0,\dots, 0)\in\Z^n$.

Then for any $g\in U(n)$ such that
$g\Lambda_{v}\cap\Lambda_{w}=\0$ and for any $k\in\N$,
there exists $C_k>0$ such that for all $p\in\N$, we have
\begin{equation}\label{mainflaexptoric}
\left|\<ge_{pv_p},e_{pw_p}\>\right|\leq C_k\,p^{-k}\,.
\end{equation}

For any $g\in U(n)$ such that the intersection,
$g\Lambda_{v}\cap\Lambda_{w}$ is clean,
one can choose the basis vectors so that for each $0\leq q\leq m$,
there exist $\kappa_q\in\Z/4\Z$ and $b^{(q)}\in\C$
such that the following asymptotic expansion holds
as $p\fl+\infty$,
\begin{equation}\label{<e1e2>}
\<g e_{pv_p},e_{pw_p}\>
=z_p\,p^{-\frac{\dim\Lambda_v+\dim\Lambda_w}{4}}
\sum_{q=0}^m p^{\frac{\dim Y_q}{2}}
\sqrt{-1}^{\kappa_q}e^{2\pi\sqrt{-1}p\eta_{p}^{(q)}}
\left(b_{p}^{(q)}+O(p^{-1})\right)\,,
\end{equation}
with $b_p^{(q)}\to b^{(q)}$ as $p\to+\infty$.


In case $n\in\N^*$ is even, considering
a weight basis associated with $\Gamma_{p-\frac{n}{2}}$
instead of $\Gamma_p$ with
$(v_p,\,w_p\in \frac{1}{p}\,\Gamma_{p-\frac{n}{2}})_{p\in\N}$,
and assuming that $v,\,w\in\Delta^0\subset\Delta$,
then for any $g\in U(n)$ such that the intersection
$g\Lambda_v\cap\Lambda_w$ is
transverse, we have the following asymptotic expansion
as $p\to+\infty$,
\begin{equation}\label{toricfla}
p^{\frac{n-1}{2}}\<g e_{pv_p},e_{pw_p}\>
=\sum_{q=0}^m \sqrt{-1}^{\kappa_q}
\frac{e^{2\pi\sqrt{-1}p\eta_{p}^{(q)}}}{
\big|\det\left(\alpha^{(q)}_p([\Ad_g\xi_j,\xi_k])\right)_{j,\,k}\big|^{\frac{1}{2}}}+O(p^{-1})\,,
\end{equation}
where $\{\xi_j\in\t_0\}_{j=1}^{n-1}$ is a basis of the integral lattice
inside the Lie subalgebra $\t_0\subset\u(n)$ of $T_0\subset U(n)$ and where
$\alpha^{(q)}_p\in\u(n)^*$ is such that $Y^{(q)}_p=\{\alpha^{(q)}_p\}$ for all
$0\leq q\leq m$.
\end{theorem}

\cref{toricth} is proved in \cref{prooftoricth}.
Note that every irreducible representation of $SU(n)\subset U(n)$
can be obtained from
an irreducible representation of $U(n)$ by restriction,
and that the weight basis considered in \cref{toricth} induces by restriction
to $SU(n)\subset U(n)$ a weight basis for the maximal torus of $SU(n)$, so that
\cref{toricth} can also be stated in terms of irreducible
representations of $SU(n)$.
In the case $n=2$, the irreducible representation of $U(2)$ of highest weight
$(p,0)$ induces the \emph{spin-}$\frac{p}{2}$ representation
of $SU(2)\subset U(2)$, and the elements of a weight basis
for the $1$-dimensional torus action $T_0\subset U(n)$
coincide with the standard \emph{spin states} in quantum mechanics.
The matrix elements computed in \cref{toricth} then coincide
with \emph{Wigner's $d$-matrix elements}, as described for instance in
\cite[\S\,3.1]{LY09}. Furthermore, the associated coadjoint orbit
$(X_\lambda,\om)$ is naturally identified with the $2$-dimensional
sphere $S^2$ endowed with its standard volume form, and \cref{toricth}
then recovers
the semiclassical \emph{spherical area formula} for Wigner's $d$-matrix
elements,
as stated for instance in \cite[\S\,3.2,\,(57)]{LY09}.
Here, the spherical area refers to the area delimited by the intersection of two
circles on the sphere $S^2$, which is computed by the term
$\eta_{p}^{(q)}\in\R$
in formula \cref{<e1e2>}. As explained in
\cite[\S\,4.1,\,(68)]{LY09}, the asymptotics of Wigner's $d$-matrix elements
are closely related to the asymptotics of the classical $6j$-symbols
discovered by Ponzano and Rege in \cite{PR68}, and a rigorous proof of
these asymptotics
was first given by Roberts in \cite{Rob99}, who also used methods of Geometric
Quantization by interpreting the sphere as a coadjoint orbit.

On the other hand, \cref{toricth} gives a criterion from representation theory
for the general problem of whether the intersection
$g\Lambda_{v}\cap\Lambda_{w}\subset\CP^{n-1}$ is non-empty, for any given
$v,\,w\in\Delta$.
Already in the case when $v$ is the \emph{barycenter} of $\Delta$,
so that $\Lambda_v\subset\CP^{n-1}$ is the so-called \emph{Clifford torus}
of $\CP^{n-1}$ as a toric manifold, the only known proof of the fact that
$g\Lambda_v\cap\Lambda_v\neq\0$ for all $g\in U(n)$ follows from the
Hamiltonian non-displaceability of Clifford tori due to Biran,
Entov and Polterovich in \cite{BEP04} and Cho in \cite{Cho04},
which is based on sophisticated tools of symplectic
topology.
The same were used methods by Nishinou, Nohara and Ueda in \cite[Th.\,12.1]{NNU10} to extend this result to general Gelfand-Zetlin systems as in \cref{mainth}.
We thus hope that \cref{mainth,toricth} can shed light on the non-displaceability of remarkable
Lagrangian submanifolds by automorphisms groups of Kähler manifolds.

In \cref{BTsec,groupsec}, we recall
the geometric picture behind \cref{mainth,toricth}, which
is that of \emph{Geometric Quantization},
following the initial insight of Guillemin and Sternberg
in \cite[\S\,1]{GS83}.
This picture is based on a result of Kostant,
which we recall in \cref{Lpreqprop}, stating that for any
$\lambda=(\lambda_1,\lambda_2,\dots,\lambda_n)\in\Z^n$
with
$\lambda_1\geq\lambda_2\geq\cdots\geq\lambda_n$, the coadjoint
action of $U(n)$ on the associated coadjoint orbit $X:=X_\lambda$
lifts to a natural Hermitian line bundle
$(L,h^L)$ endowed with a Hermitian connection
$\nabla^L$
satisfying the following \emph{prequantization formula},
\begin{equation}\label{preq}
\om=\frac{\sqrt{-1}}{2\pi}R^L\,,
\end{equation}
where $R^L\in\Om^2(X,\C)$ is the curvature of $\nabla^L$.
The \emph{holomorphic
quantization} of $(X,\om)$ is then given by the space of its
holomorphic sections $H^0(X,L)$, for the holomorphic
structure on $L$ induced by a natural compatible
$U(n)$-invariant complex structure on
$(X,\om)$. By the celebrated
\emph{Borel-Weil theorem},
which we recall in \cref{BWth},
the induced action of $U(n)$ then makes $H^0(X,L)$
into an irreducible representation of $U(n)$ with highest weight $\lambda$.
On the other hand, 
a submanifold $\Lambda\subset X$ is said to satisfy the
\emph{Bohr-Sommerfeld condition} if there exists
a non-vanishing section $\zeta\in\cinf(\Lambda,L|_{\Lambda})$ satisfying
\begin{equation}
\nabla^L_{\xi}\zeta\equiv 0\,\text{ for all }\xi\in\cinf(\Lambda,T\Lambda)\,.
\end{equation}
These submanifolds represent quantum states in the old
\emph{ -Sommerfeld quantization}
process, and we explain in \cref{Lagstate} how to associate
to such a manifold an element of $H^0(X,L)$, which we call the
associated \emph{isotropic state} and which we interpret
as the corresponding quantum state in holomorphic quantization.
We then explain in \cref{theonorme,theointergal} how the Hermitian
product of two isotropic states associated with two submanifolds
$\Lambda_1$ and $\Lambda_2$ satisfying the Bohr-Sommerfeld condition can be computed
in terms of the geometry of the intersection $\Lambda_1\cap\Lambda_2$
at the semiclassical limit $p\to+\infty$, where one replaces $L$
by its $p^{\text{th}}$ tensor power $L^p$. This is based on results of the author in \cite{Ioo18b}, following the seminal work
of Borthwick, Paul and Uribe in \cite[Th.\,3.2]{BPU98},
and uses tools of \emph{Berezin-Toeplitz quantization}
first developed by Boutet de Monvel and Sjöstrand
\cite{BdMS75}
and Boutet de Monvel and Guillemin \cite{BdMG81}, while the author
follows in \cite{Ioo18b} the approach due to Ma and Marinescu in
\cite{MM07,MM08b}. The modern approach to Berezin-Toeplitz quantization
is based on a classical result of Bordemann, Meinrenken and Schlichenmaier
in \citep{BMS94}, which we recall in \cref{BMS},
establishing the quantum-classical correspondence between classical observables,
represented by smooth function $\cinf(X,\R)$ over $(X,\om)$,
and quantum observables, represented by Hermitian operators acting
on $H^0(X,L^p)$, at the semiclassical limit as $p\to+\infty$.

In \cref{quantprojsec}, we consider first
the projective case of $X=\CP^{n-1}$ treated in
\cref{toricth}, in which case the prequantizing line bundle $(L,h^{L})$
coincides with the
\emph{dual of the tautological line bundle} over $\CP^{n-1}$
endowed with its associated Fubini-Study metric. Then for any $p\in\N$, the
space of holomorphic sections $H^0(X,L^p)$ naturally identifies with
the space of homogeneous polynomials of degree $p\in\N$ over $\C^n$,
which is the classical
realization of the irreducible representation of $U(n)$ with highest weight
$(p,0,\dots, 0)\in\Z^n$ constructed by Weyl, as explained for instance in \cite[\S\,6.1]{FH91}.
%
On the other hand, from general considerations on toric manifolds,
the fibres of the moment map \cref{Mtoric} satisfying the Bohr-Sommerfeld condition
are exactly those above integral points 
inside $p\Delta\subset\R^{n-1}$, where $\Delta$ is the Delzant polytope
of $\CP^{n-1}$,
and we show in \cref{BStoric} that the associated isotropic states
provide basis elements for the weight decomposition
of $H^0(X,L^p)$ for the torus action $T_0\subset U(n)$ with corresponding
weight. This correspondence holds
as such for all toric manifolds, and provides the most precise instance of
the expected correspondence between holomorphic quantization
and Bohr-Sommerfeld quantization.
The proof of \cref{toricth}, that we give in \cref{prooftoricth},
then follows as an
application of the results of the author in \cite{Ioo18b} described
above. Note that the previous result of Borthwick, Paul and Uribe in
\cite{BPU98} only holds for Bohr-Sommerfeld Lagrangian submanifolds,
which corresponds
to the case when $w\in\Delta^0\subset\Delta$ in \cref{toricth}, while the
case of $w\in\Delta$ belonging to a general face requires the more general
case of Bohr-Sommerfeld isotropic submanifolds considered in \cite{Ioo18b},
since in that case
the dimension of the fibre $\Lambda_v\subset\CP^{n-1}$ is strictly lower than $n-1$.
On the other hand, Borthwick, Paul and Uribe compute in \cite[Th.\,4.4]{BPU98}
a special case of their formula when $\dim_\C X=1$, from which 
one readily recovers the spherical area formula for Wigner's
$d$-matrix elements associated
with great circles of $S^2$, as explained in \cite[Th.\,16]{BN23}.

In \cref{GZquantsec}, we proceed to consider the case of a general
regular weight $\lambda=(\lambda_1,\lambda_2,\dots,\lambda_n)\in\Z^n$ with
$\lambda_1>\lambda_2>\cdots>\lambda_n$, so that the associated coadjoint
orbit $X_\lambda\subset\u(n)^*$ has maximal dimension among coadjoint
orbits in $\u(n)^*$. In \cref{GZbassec}, we
introduce the tools of representation theory
of compact Lie groups needed for 
the proof of \cref{mainth}. In particular, the Gelfand-Zetlin decomposition
\cref{GZdecflaintro} is the decomposition into common eigenspaces of the
commutative
\emph{Gelfand-Zetlin subalgebra}
\begin{equation}\label{GZalgintro}
\AA_n:=\<Z[U(\u(k))]~|~1\leq k\leq n\>\subset U(\u(n))\,,
\end{equation}
which we introduce in \cref{GZalgdef} as the commutative subalgebra
of the \emph{universal
enveloping algebra} $U(\u(n))$ of $\u(n)$ generated by the centers
$Z[U(\u(k))]\subset U(\u(k))$ of the universal
enveloping algebra of $\u(k)$ for each $k\in\N$ with $k\leq n$
via the sequence of inclusions \cref{GLinc}.
We study
the common eigenvalues of the Gelfand-Zetlin subalgebra \cref{GZalgintro}
in the decomposition \cref{GZdecflaintro}, in order to link it with the parametrization
\cref{W=DZ} due to Guillemin and Sternberg in \cite[Prop.\,5.4]{GS83}. In particular,
the half-sum of roots behind the $\rho$-shift \cref{rhobar} plays an essential role
in the \emph{Harish-Chandra isomorphism} used to determine these
eigenvalues, which we recall in \cref{HCth}. In \cref{proofmainthsec}, we then give the proof of
\cref{mainth}, by showing that a natural generating family for the
Gelfand-Zetlin subalgebra \cref{GZalgintro} can be expressed
in terms of the Berezin-Toeplitz quantization of the Gelfand-Zetlin system
\cref{GZsystintro} over $(X_\lambda,\om)$ introduced by Guillemin
and Sternberg in \cite{GS83}, then using results of
Charles in \cite{Cha03b,Cha06} on the semiclassical limit of quantum
integrable systems
in Berezin-Toeplitz quantization, to show that Gelfand-Zetlin bases elements
are Lagrangian states associated with regular Bohr-Sommerfeld
fibres of the Gelfand-Zetlin system \cref{GZsystintro} in an appropriate sense.
Note that the $\rho$-shift \cref{rhobar}
corresponds here to the
\emph{metaplectic correction} appearing in \cite{Cha06}, which corresponds
to the choice of a square root of the \emph{canonical line bundle}
\cref{KX} of $X_\lambda$, and which is needed in order to provide the purely
symplectic formula \cref{mainfla} for the first order term of the asymptotics
of matrix elements established in the main \cref{mainth}.
Note on the other hand that the projective space
$\CP^{n-1}$ considered
in \cref{toricth} admits a metaplectic correction if and only if
$n\in\N^*$ is even, which explains the shift by $n/2$ needed to obtain
the purely symplectic formula \cref{toricfla} in the second main
\cref{toricth}.

Let us finally point out that one of our motivation for \cref{mainth}
comes from
the study of $SU(2)$-character varieties of surfaces, which
admit a remarkable classical integrable system due to Jeffrey and Weitsman
in \cite{JW92},
as well as a natural action of the mapping class group of the surface.
On the other hand, character varieties admit a natural quantization
endowed with a natural action of the mapping class group,
provided by the TQFT of Witten in \cite{Wit89}, Reshetikhin and Turaev
in \cite{RT91} and Blanchet, Habegger, Masbaum and Vogel in \cite{BHMV95}.
This quantization also comes with a natural basis, and the
matrix elements for the mapping class group action in this basis
have been computed by Detcherry in
\cite[Th.\,11.1]{Det18} using the results of Charles in \cite{Cha03b,Cha06},
obtaining the asymptotics \cref{mainfla} in this setting. Note also that 
the same
method was used by Charles in \cite[Th.\,7.1]{Cha10b} to compute
the asymptotics of the classical $6j$-symbols, recovering the result
of Roberts in \cite{Rob99}.

\begin{ackn*}
The author came to the idea behind this work after
extended conversations with Ramanujan Santharoubane,
and wishes to thank him first
and foremost. The author also wants to thank Teodor Banica, Olivier Biquard, Raïka Dehy
and Michela Varagnolo for useful discussions. Finally, the author thanks the anonymous referees for suggesting improvements to the manuscript.
This project was partially supported by the ANR-23-CE40-0021-01 JCJC
project QCM.
\end{ackn*}
 
\section{Geometric quantization}
\label{BTsec}

In this Section, we describe the general set-up of Geometric quantization
used in the proofs of \cref{mainth,toricth}, and in particular that of
Berezin-Toeplitz quantization, following \cite{Ioo18b} and \cite{MM07}.

\subsection{Berezin-Toeplitz quantization}
\label{BTsubsec}

Let $(X,\om)$ be a compact symplectic manifold of dimension $2n\in\N^*$
together with a Hermitian line bundle $(L,h^L)$ endowed with a
Hermitian
connection $\nabla^L$ satisfying the prequantization formula \cref{preq}.
Let also $X$ be equipped with a
complex structure $J\in\End(TX)$ compatible with $\om$, making $(X,\om,J)$
into a \emph{Kähler manifold},
and write $g^{TX}$ for the \emph{Kähler metric} of $(X,\om,J)$, which
is the Riemannian metric defined over $X$
by the formula
\begin{equation}\label{gTX}
g^{TX}(\cdot,\cdot)=\om(\cdot,J\cdot)\,.
\end{equation}
The Riemannian volume form $dv_X$ of $(X,g^{TX})$ then coincides with the
\emph{Liouville volume form} of $(X,\om)$, that is
\begin{equation}\label{voldef}
dv_X=\frac{\om^n}{n!}\,.
\end{equation}
Let us write
\begin{equation}\label{splitc}
T_\C X=T^{(1,0)}X\oplus T^{(0,1)}X
\end{equation}
for the splitting of the complexification $T_\C X$ of the tangent bundle
of $X$ into the eigenspaces of $J$ corresponding to the eigenvalues $\sqrt{-1}$ and $-\sqrt{-1}$ respectively. For any $v\in\cinf(X,TX)$, we write
$v=v^{(1,0)}+v^{(0,1)}$ for its decomposition into this splitting.
Under the canonical isomorphism of complex vector bundles
$(TX,J)\simeq T^{(1,0)}X$ induced by the splitting \cref{splitc},
the Hermitian metric on $T^{(1,0)}X$ induced by the restriction of $g^{TX}$
via \cref{splitc} corresponds to the Hermitian metric $h^{TX}$ on $(TX,J)$
defined by
\begin{equation}\label{HermTX}
h^{TX}=g^{TX}-\sqrt{-1}\om\,.
\end{equation}
The \emph{canonical line bundle} of $(X,J,\om)$ is the holomorphic line bundle
\begin{equation}\label{KX}
K_X=\det(T^{(1,0)}X^*)
\end{equation}
endowed with the Hermitian structure $h^{K_X}$ and
Hermitian connection $\nabla^{K_X}$ respectively induced by $g^{TX}$ and
$\nabla^{TX}$ via the splitting \cref{splitc}.
By the prequantization formula \cref{preq},
the curvature $R^L\in\Om^2(X,\C)$ of $\nabla^L$ is
$J$-invariant, which implies that
the $(0,1)$-part of $\nabla^L$ under the splitting \cref{splitc} defines a
Cauchy-Riemann $\dbar$-operator, inducing a holomorphic structure on $L$.

For any holomorphic Hermitian line bundle $(K,h^K)$ over $X$,
we write $|\cdot|_{h^K}$
for the pointwise norm on $K$ induced by $h^K$. For any $p\in\N$,
write $L^p:=L^{\otimes p}$ for the $p^{\text{th}}$ tensor power of $L$,
and $h^{p},\,\nabla^{p}$
for the Hermitian metric and connection on the tensor product
$L^p\otimes K$ respectively induced 
by $h^L,\,h^K$ and $\nabla^L,\,\nabla^K$.
We denote by $\cinf(X,L^p\otimes K)$ the space of smooth sections of
$L^p\otimes K$, endowed with the \emph{$L^2$-Hermitian product}
$\<\cdot,\cdot\>_{L^2}$ given for any $s_1,s_2\in\cinf(X,L^p\otimes K)$ by the formula
\begin{equation}\label{L2}
\<s_1,s_2\>_{L^2}=\int_X h^p(s_1(x),s_2(x))\,dv_X(x)\,.
\end{equation}
We write $\|\cdot\|_{L^2}$ for the associated $L^2$-norm, and
$L^2(X,L^p\otimes K)$ for the completion of $\cinf(X,L^p\otimes K)$
with respect to $\|\cdot\|_{L^2}$. We also write $\|\cdot\|_{op}$ for the
operator norm on bounded operators acting on $L^2(X,L^p\otimes K)$.

Following for instance \cite[Th.\,1.4.1]{MM07}, the subspace
$H^0(X,L^p\otimes K)\subset L^2(X,L^p\otimes K)$
of holomorphic sections of $L^p\otimes K$ is finite-dimensional, so that
the orthogonal projection
\begin{equation}\label{projdef}
P_p:L^2(X,L^p\otimes K)\longrightarrow H^0(X,L^p\otimes K)
\end{equation}
with respect to the $L^2$-product \cref{L2} has finite rank, and
hence admits a smooth Schwartz kernel.
In the following definition, we introduce a fundamental tool of this paper.

\begin{defi}
\label{BTdef}
For any $f\in\cinf(X,\R)$ and $p\in\N$, the associated
\emph{Berezin-Toeplitz operator} is the operator $T_p(f)$ acting on
$L^2(X,L^p\otimes K)$ defined by
\begin{equation}
T_p(f):=P_p\,f\,P_p\,,
\end{equation}
where $f$ denotes the operator of multiplication by $f$.
\end{defi}

The main interest of Berezin-Toeplitz operators is that they provide
\emph{quantum observables} for the holomorphic quantization
of the symplectic manifold $(X,\om)$. This is described
by the following fundamental theorem due to Bordemann, Meinrenken and Schlichenmaier
\cite{BMS94},
which we present here in a guise due to Ma and Marinescu in \cite{MM08b}.

\begin{theorem}\label{BMS}
{\cite{BMS94,MM08b}}
For any $f\in\cinf(X,L^p\otimes L)$, we have
\begin{equation}\label{BMS0}
\|T_{p}(f)\|_{op}\xrightarrow{p\to+\infty}|f|_{\CC^0}\,,
\end{equation}
where $|f|_{\CC^0}$ denotes the uniform norm of $f$.

Furthermore, for any $f,\,g\in\cinf(X,L^p\otimes K)$, we have the following
estimates in operator norm as $p\to+\infty$,
\begin{equation}\label{BMS1}
T_p(f)T_p(g)=T_p(fg)+O(p^{-1})\,,
\end{equation}
and
\begin{equation}\label{BMS2}
[T_{p}(f),T_{p}(g)]=
\frac{\sqrt{-1}}{2\pi p}T_{p}(\{f,g\})+O(p^{-2})\,,
\end{equation}
where $\{\cdot,\cdot\}$ denotes the Poisson bracket on $\cinf(X,\R)$
induced by $\om$.
\end{theorem}

The approach of Ma and Marinescu in \cite{MM08b} to establish \cref{BMS} is based
on the off-diagonal asymptotic expansion as $p\to+\infty$ of the
smooth Schwarz kernel of the orthogonal projection \cref{projdef},
which we introduce in the following definition.

\begin{defi}\label{Bergdef}
For any $p\in\N$, the \emph{Bergman kernel}
$P_p(x,y)\in(L^p\otimes K)_x\,\otimes\,(L^p\otimes K)_y^*$, for all $x,\,y\in X$,
is the Schwarz kernel of the orthogonal projection \cref{projdef},
characterized for any
$s\in\cinf(X,L^p\otimes K)$ and $x\in X$ by
the formula
\begin{equation}\label{ker}
(P_p s)(x)=\int_X P_p(x,y).s(y)\,dv_X(y).
\end{equation}
\end{defi}

\subsection{Isotropic states}
\label{isosec}

In this Section, we describe the semiclassical asymptotics
of \emph{isotropic states} in Berezin-Toeplitz quantization, which are
at the basis of the asymptotics described in \cref{mainth,toricth}.

Recall that a properly embedded submanifold $\iota:\Lambda\hookrightarrow X$
in a symplectic manifold
$(X,\om)$ is said to be \emph{isotropic} if $\iota^*\om=0$.
If in addition $\dim\Lambda=n$, it is said to be \emph{Lagrangian}.
We write
$dv_\Lambda$ for
the Riemannian volume form of $(\Lambda,\iota^*g^{TX})$.

Let $\nabla^{\iota^*L},\,h^{\iota^*L}$ be the connection and Hermitian metric induced by $\nabla^L,\,h^L$ on the pullback line bundle $\iota^*L$ over $\Lambda$. Note that by \cref{preq}, the condition $\iota^*\om=0$ implies that $\nabla^{\iota^*L}$ is \emph{flat}. This observation motivates the following definition.

\begin{defi}\label{BS}
For any $p\in\N$,
a properly embedded oriented submanifold
$\iota:\Lambda\hookrightarrow X$ is said to
satisfy the \emph{Bohr-Sommerfeld condition} at level $p\in\N$ if there exists a non-vanishing
smooth section $\zeta^p\in\cinf(\Lambda,\iota^* L^p)$ satisfying
\begin{equation}\label{nabs=0}
\nabla^{\iota^*L^p} \zeta^p=0\,.
\end{equation}
A \emph{sequence of Bohr-Sommerfeld submanifolds} is a sequence
$\{(\Lambda_p,\zeta^p,f_p)\}_{p\in\N}$ of
submanifolds $\iota_p:\Lambda_p\hookrightarrow X$ satisfying
the Bohr-Sommerfeld condition at level $p\in\N$,
of sections $\zeta^p\in\cinf(\Lambda_p,\iota^*_p L^p)$
with $|\zeta^p|_{\iota^*_p L^p}\equiv 1$
satisfying \cref{nabs=0}, and of sections
$f_p\in\cinf(\Lambda_p,\iota^*_p K)$,
for all $p\in\N$.
%
%
\end{defi}

In the case $K=\C$ and $f_p=1$ for all $p\in\N$,
we will simply write $\{(\Lambda_p,\zeta^p)\}_{p\in\N}$ for a sequence
of Bohr-Sommerfeld submanifolds. When $\dim\Lambda_p=n$ for all $p\in\N$,
we talk about a sequence of \emph{Bohr-Sommerfeld Lagrangian submanifolds}.

As explained in the Introduction,
a Bohr-Sommerfeld submanifold represents a quantum state
in the old Bohr-Sommerfeld quantization scheme, and there should be a
corresponding state
in the holomorphic quantization scheme described in this Section.
Following \cite[(12)]{BPU98} and \cite[Def.\,3.3]{Ioo18b}, this is provided by
the following definition.

\begin{defi}\label{Lagstate} The \emph{isotropic state} associated with 
a sequence of Bohr-Sommerfeld submanifolds $\{(\Lambda_p,\zeta^p,f_p)\}_{p\in\N}$
is the sequence of sections $\{s_{\Lambda_p}\in H^0(X,L^p\otimes K)\}_{p\in\N}$
defined for any $x\in X$ by the formula
\begin{equation}\label{defLagstate}
s_{\Lambda_p}(x)=\int_{\Lambda_p} P_p(x,y).\zeta^pf_p(y)\,dv_{\Lambda_p}(y).
\end{equation}
\end{defi}

An isotropic state associated with a sequence of Bohr-Sommerfeld Lagrangian
submanifolds is called a \emph{Lagrangian state}.
Isotropic states are characterized by the following reproducing property,
which readily follows from their \cref{Lagstate}.

\begin{prop}\label{proprepgal}
{\cite[Prop.\,3.4]{Ioo18b}}
For any $p\in\N$ and $s\in H^0(X,L^p\otimes K)$, we have
\begin{equation}\label{rep}
\<s,s_{\Lambda_p}\>_{L^2}=\int_{\Lambda_p}
h^p(s(x),\zeta^pf_p(x))\,dv_{\Lambda_p}(x).
\end{equation}
\end{prop}
\begin{proof}
By \cref{Bergdef} of the Bergman kernel, the fact that the
orthogonal projection
\cref{projdef} is self-adjoint translates into the following formula,
for any $x,\,y\in X$ and
$\eta\in\cinf(X,L^p\otimes K)$,
\begin{equation}
h^p\left(\eta(x),P_p(x,y).\eta(y)\right)
=h^p\left(P_p(y,x).\eta(x),\eta(y)\right)\,.
\end{equation}
Furthermore, for any $s\in H^0(X,L^p\otimes K)$, the fact that the orthogonal
projection \cref{projdef} acts
as the identity on $H^0(X,L^p)$ translates into the following reproducing
formula, for all $x\in X$,
\begin{equation}
s(x)=\int_X\,P_p(x,y).s(y)\,dv_X(y)\,.
\end{equation}
Then using \cref{Lagstate} and Fubini's theorem, we compute
\begin{equation}\label{comprep}
\begin{split}
\<s,s_{\Lambda_p}\>_{L^2} & =\int_X\int_{\Lambda_p}\,h^p\left(s(y), P_p(y,x).\zeta^pf_p(x)\right)\,dv_{\Lambda_p}(x)\,dv_X(y)\\
& =\int_{\Lambda_p}\int_X\,h^p\left(P_p(x,y).s(y),\zeta^pf_p(x)\right)\,dv_X(y)\,dv_{\Lambda_p}(x)\\
& =\int_{\Lambda_p} h^p(s(x),\zeta^pf_p(x))\ dv_{\Lambda_p}(x).
\end{split}
\end{equation}
This shows the result.
\end{proof}

The following Definition is purely technical, and will be used to deal with
the varying sequences of Bohr-Sommerfeld submanifolds appearing in the
main \cref{mainth,toricth}.

\begin{defi}\label{cvsmoothdef}
We say that a sequence $\{\iota_p:\Lambda_p\hookrightarrow X\}_{p\in\N}$ of
embedded submanifolds of $X$
\emph{converges smoothly} towards a properly embedded submanifold
$\iota:\Lambda\hookrightarrow X$ if there exists
$p_0\in\N$ and a finite collection of charts
$\{U_j\subset X\}_{j=1}^m$ covering both $\Lambda$ and $\Lambda_{p}$ for all $p\geq p_0$, together with
diffeomorphisms $\phi_j:U_j\longrightarrow V\subset\R^{2n}$ 
with $\phi_j(U_j\cap\Lambda)=V\cap\R^{\dim\Lambda}$
and diffeomorphisms $\phi_j^{(p)}:U_j\longrightarrow V\subset\R^{2n}$
with $\phi_j^{(p)}(U_j\cap\Lambda_{p})=V\cap\R^{\dim\Lambda}$, for all
$1\leq j\leq m$ and $p\geq p_0$, such that $\phi_j^{(p)}\circ\phi_j^{-1}:V\to V$ converges smoothly towards the
identity as $p\to+\infty$.

We say that a sequence $\{(\Lambda_p,\zeta^p,f_p)\}_{p\in\N}$
of Bohr-Sommerfeld submanifolds
\emph{converges smoothly} towards the couple $(\Lambda,f)$,
where $f\in\cinf(\Lambda,\iota^*K)$ is a section over $\Lambda$,
if furthermore,
there exist $f_j\in\cinf(U_j,K)$ for all $1\leq j\leq m$
such that $f|_{\Lambda\cap U_j}=\iota^*f_j$ and
$f_p|_{\Lambda\cap U_j}=\iota^*_pf_j$ for all $p\geq p_0$.
\end{defi}

Note that if a sequence of submanifolds $\{\Lambda_p\subset X\}_{p\in\N}$
converges smoothly towards a compact submanifold $\Lambda\subset X$, then
$\Lambda_p$ is diffeomorphic to $\Lambda$ for all $p\in\N$ large enough.
We will also need the following standard notion.

\begin{defi}\label{cleandef}
We say that two submanifolds $\Lambda^{(1)},\,\Lambda^{(2)}\subset X$
are intersecting \emph{cleanly} if the intersection
$\Lambda^{(1)}\cap\Lambda^{(2)}$ is a submanifold of $X$
such that for any $x\in\Lambda^{(1)}\cap\Lambda^{(2)}$, we have
$T_x \Lambda^{(1)}\cap T_x\Lambda^{(2)}
=T_x(\Lambda^{(1)}\cap\Lambda^{(2)})$. 
\end{defi}

Note that if two respective sequences of submanifolds
$\{\Lambda^{(1)}_p,\,\Lambda^{(2)}_p\subset X\}_{p\in\N}$
converge
smoothly towards
$\Lambda^{(1)},\,\Lambda^{(2)}\subset X$ intersecting cleanly,
then the sequence
$\{\Lambda^{(1)}_p\cap\Lambda^{(2)}_p\}_{p\in\N}$
converges smoothly towards $\Lambda^{(1)}\cap\Lambda^{(2)}$.




The following result, adapted from \cite{Ioo18b},
gives an asymptotic estimate on the norm of
an isotropic state in terms of the volume of the
associated isotropic submanifold as $p\fl+\infty$. 

\begin{theorem}\label{theonorme}
{ \cite[Th.\,3.6]{Ioo18b}}
Let $\iota:\Lambda\hookrightarrow X$ be an isotropic submanifold
endowed with $f\in\cinf(\Lambda,\iota^*K)$, and let
$\{(\Lambda_p,\zeta^p,f_p)\}_{p\in\N}$ be a sequence of Bohr-Sommerfeld submanifolds
converging smoothly towards $(\Lambda,f)$.
Then
the $L^2$-norm of the associated isotropic state
$\{s_{\Lambda_p}\in H^0(X,L^p\otimes K)\}_{p\in\N}$
satisfy the following asymptotic
expansion as $p\fl+\infty$,
\begin{equation}\label{norme}
\big\|s_{\Lambda_p}\big\|^2_{L^2}=2^{\frac{\dim\Lambda}{2}}p^{n-\frac{\dim\Lambda}{2}}\left(
\int_{\Lambda}|f_p|_{K}^2\,dv_{\Lambda_p}
+O(p^{-1})\right)\,.
\end{equation}
\end{theorem}
\begin{proof}
This is a straightforward consequence of the proof of
\cite[Th.\,3.6]{Ioo18b}, where all estimates are uniform with respect
to deformation of parameters, so that it
readily extends to cover the case of sequences of smoothly converging
submanifolds in the sense of \cref{cvsmoothdef} instead of a fixed
Bohr-Sommerfeld submanifold for all $p\in\N$.
\end{proof}

The following result, also adapted from \cite{Ioo18b},
gives an asymptotic expansion as $p\fl+\infty$
on the Hermitian product of two isotropic states in terms of the intersection
of the associated isotropic submanifolds.

\begin{theorem}\label{theointergal}
{\cite[Th.\,4.3]{Ioo18b}}
Let $\iota_j:\Lambda^{(j)}\hookrightarrow X$ be isotropic submanifolds
endowed with $f^{(j)}\in\cinf(\Lambda^{(j)},\iota_j^*K)$,
let
$\{(\Lambda_{p}^{(j)},\zeta_j^p,f_{p}^{(j)})\}_{p\in\N}$
be sequences of Bohr-Sommerfeld submanifolds
converging
smoothly towards $(\Lambda^{(j)},f^{(j)})$,
and write
$\{s_{\Lambda_p^{(j)}}\in H^0(X,L^p\otimes K)\}_{p\in\N}$ for the
associated isotropic states, for each $j=1,\,2$.

If $\Lambda_1\cap\Lambda_2=\0$, then for any $k\in\N$,
there exists $C_k>0$ such that for all $p\in\N$, we have
\begin{equation}\label{theointergalfla}
\big|\big\langle s_{\Lambda_p^{(1)}},s_{\Lambda_p^{(2)}}\big\rangle_{L^2}\big|
\leq C_k\,p^{-k}\,.
\end{equation}

If $\Lambda_1$ and $\Lambda_2$ intersect cleanly,
writing $\Lambda^{(1)}\cap\Lambda^{(2)}=\cup_{q=0}^m Y^{(q)}$
for the decompositions into connected components,
there exist $b_{p}^{(q)}\in\C$ for each $0\leq q \leq m$ and $p\in\N$,
such that for any $k\in\N$ and as $p\fl+\infty$,
\begin{equation}\label{<u1u2>}
\big\langle s_{\Lambda_p^{(1)}},s_{\Lambda_p^{(2)}}\big\rangle_{L^2}=
 p^{n-\big(\frac{\dim\Lambda^{(1)}}{2}+\frac{\dim\Lambda^{(2)}}{2}\big)}
\sum_{q=0}^m p^{\frac{\dim Y^{(q)}}{2}}
\lambda^{(q)}_p
\left(b_{p}^{(q)}+O(p^{-1})\right)\,,
\end{equation}
where
%
%
$\lambda^{(q)}_p\in\C$ is the value of the constant function on $Y_{p}^{(q)}$
defined for any $x\in Y^{(q)}_p$ by
$\lambda^{(q)}_p(x)=h^{L^p}(\zeta_1^p(x),\zeta_2^p(x))$,
with
$\Lambda_{p}^{(1)}\cap\Lambda_{p}^{(2)}=\cup_{q=0}^m Y_{p}^{(q)}$ a
decomposition into connected components for all $p\in\N$
such that
$\{Y_{p}^{(q)}\}_{p\in\N}$ converges smoothly towards $Y^{(q)}$ for each
$1\leq q\leq m$.

Furthermore, if $\dim\Lambda_1=n$, we have
\begin{multline}\label{bm0}
b_{p}^{(q)}=2^{n/2}\int_{Y^{(q)}_p} h^K\left(f_p^{(1)},f_p^{(2)}\right)\\
\det{}^{-\frac{1}{2}} \left(\sum_{r=1}^{n-\dim Y^{(q)}}\sqrt{-1}\,h^{TX}(e^{(2)}_i,e_r^{(1)})\,\om(e_k^{(2)},e_r^{(1)})\right)_{i,k}|dv|_{Y^{(q)}},
\end{multline}
for some choice of square root for the determinant,
where $\{e_k^{(j)}\}_k$
are local orthonormal frames of the normal bundle of $Y^{(q)}_p$
inside $\Lambda_{p}^{(j)}$ for each $j=1,\,2$, and
$|dv|_{Y^{(q)}_p}$ is the Riemannian density 
associated with the restriction of $g^{TX}$ to $Y^{(q)}_p$.
\end{theorem}
\begin{proof}
This is a straightforward consequence of the proof of
\cite[Th.\,4.3]{Ioo18b}, where all estimates are uniform with respect
to deformations of parameters, so that it
readily extends to cover the case of sequences of converging submanifolds
in the sense of \cref{cvsmoothdef} instead of a fixed Bohr-Sommerfeld manifold
for all $p\in\N$.$^1$
\end{proof}

\footnotetext[1]{Note however that the determinant term in the integrand
of formula \cref{bm0} is actually the complex conjugate of the corresponding term
in \cite[Th.\,4.2,\,(4.4)]{Ioo18b}, which is due to the wrong sign appearing
in front of the $\sqrt{-1}$-term in the second equality of
\cite[(4.17)]{Ioo18b}.
Then \cite[Th.\,6.3,\,(6.9)]{Ioo18b} should also be corrected accordingly with
a minus sign in front of $\sqrt{-1}$ appearing inside
the exponential term.}

Let us now consider the important special case when one takes $K=K_X^{1/2}$
to be a square root of the canonical line bundle \cref{KX} of $X$.
As explained in \cite[\S\,D]{LM89},
such a square root exists only if the first Chern class $c_1(TX)$ of $X$
is even in $H^2(X,\Z)$, and is not unique in general.
We endow $K_X^{1/2}$ with the Hermitian structure induced by the Hermitian
metric \cref{HermTX}. Now if $\iota:\Lambda\hookrightarrow X$ is a
properly embedded Lagrangian submanifold, we have
a canonical isomorphism
\begin{equation}\label{isoLKX}
\begin{split}
\iota^*K_X&\xrightarrow{~\sim~}\det(T^*_\C\Lambda)\\
\beta&\longmapsto\big[e_1\wedge\cdots\wedge e_n\mapsto
\beta(e_1^{(1,0)},\dots,e_n^{(1,0)})\big]\,.
\end{split}
\end{equation}
We then have the following consequence of
\cref{theointergal}.

\begin{theorem}\label{corintergal}
Let $\iota_j:\Lambda^{(j)}\to X$ be Lagrangian submanifolds
intersecting transversally
endowed with $f^{(j)}\in\cinf(\Lambda_j,\iota_j^*K_X^{1/2})$, and let
$\{(\Lambda_{p}^{(j)},\zeta_j^p,f_{p}^{(j)})\}_{p\in\N}$
be sequences of Bohr-Sommerfeld submanifolds
converging
smoothly towards $(\Lambda^{(j)},f^{(j)})$, for each $j=1,\,2$.
Then as $p\to+\infty$, we have
\begin{multline}\label{cor<u1u2>}
\big\langle s_{\Lambda_p^{(1)}},s_{\Lambda_p^{(2)}}\big\rangle_{L^2}=
2^{\frac{n}{2}}
\sum_{x\in\Lambda_p^{(1)}\cap\Lambda_p^{(2)}} 
h^{L^p}(\zeta_1^p(x),\zeta_2^p(x))\\
\det{}^{-1/2}_x\left(\sqrt{-1}\om(\xi_i^{(2)},\xi_k^{(1)})
\right)_{i,k}+O(p^{-1})\,,
\end{multline}
for some choices of square root of the determinant,
where for all $p\in\N$ large enough,
all $x\in\Lambda_p^{(1)}\cap\Lambda_p^{(2)}$ and
each $j=1,\,2$, the vectors
$\{\xi_k^{(j)}\}_{k=1}^n$ are bases of
$T_{x}\Lambda^{(j)}$ satisfying
\begin{equation}\label{fxi=1}
(f^{(j)}_p)^2(\xi_1^{(1)},\dots, \xi_n^{(1)})=1\,,
\end{equation}
via the isomorphism
\cref{isoLKX}.
\end{theorem}
\begin{proof}
Set $K=K_X^{1/2}$, and 
recall that the Hermitian metric \cref{HermTX} on $(TX,J)$
coincides with the metric on $T^{(1,0)}X$ induced by $g^{TX}$ via the
splitting \cref{splitc}.
Letting $x\in\Lambda_p^{(1)}\cap\Lambda_p^{(2)}$, if $\{\xi_k^{(j)}\}_{k=1}^n$
is a basis of $T_{x}\Lambda^{(j)}$ satisfying \cref{fxi=1}
and $\{e_k^{(j)}\}_{i=1}^{n}$ is an orthonormal basis of
$(T_x\Lambda^{(j)},\iota_j^*g^{TX})$
for each $j=1,\,2$,
we get
\begin{multline}\label{anglecomput}
h^{K}_x(f_{p}^{(1)},f_{p}^{(2)})
=\det{}^{-1/2}\left(h^{TX}_x(\xi_i^{(1)},\xi_k^{(2)})\right)_{i,k}\\
=\det{}^{-1/2}\left(h^{TX}_x(e_i^{(1)},e_k^{(2)})\right)_{i,k}
\det{}^{-1/2}(g^{TX}_x(e_i^{(1)},\xi_k^{(1)}))_{i,k}\\
\det{}^{-1/2}(g^{TX}_x(e_i^{(2)},\xi_k^{(2)}))_{i,k}\,.
\end{multline}
Thus under the assumption that $\dim\Lambda_1=\dim\Lambda_2=n$
and
using the multiplicativity property of the determinant,
we get
\begin{multline}
h^K_x(f_{p}^{(1)},f_{p}^{(2)})\,
\det{}^{-\frac{1}{2}} \Big(\sum_{r=1}^{n-\dim Y^{(q)}}\sqrt{-1}\,h^{TX}_x(e^{(2)}_i,e_r^{(1)})\,\om_x(e_k^{(2)},e_r^{(1)})\Big)_{i,k}\\
=\left|\det{}^{-1/2}\left(h^{TX}_x(e_i^{(1)},e_k^{(2)})\right)_{i,k}\right|^2
\det{}^{-1/2}(g^{TX}_x(e_i^{(1)},\xi_k^{(1)}))_{i,k}\\
\det{}^{-1/2}(g^{TX}_x(e_i^{(2)},\xi_k^{(2)}))_{i,k}
\det{}^{-\frac{1}{2}}\Big(\sqrt{-1}\om_x(e_i^{(2)},e_k^{(1)})\Big)_{i,k}\\
=\det{}^{-1/2}(\sqrt{-1}\om_x(\xi_i^{(2)},\xi_k^{(1)}))_{i,k}\,,
\end{multline}
where we also used the fact that orthonormal bases of
$(T_x\Lambda^{(j)},\iota_j^*g^{TX})$
for each $j=1,\,2$ induce orthonormal bases of $T^{(1,0)}X$ for the induced
Hermitian metric via the splitting \cref{splitc}, and that
a change of orthonormal bases
of a Hermitian vector space is unitary, hence has determinant
of complex modulus $1$.
We then get formula \cref{cor<u1u2>} from formula
\cref{bm0}.$^1$
\end{proof}

\footnotetext[1]{Note that due to the fact the determinant term appearing in the integrand
of \cite[Th.\,4.2,\,(4.4)]{Ioo18b} should be replaced by its
complex conjugate,
formula \cref{cor<u1u2>} differs from the corresponding formula
in \cite[Rmk.\,4.5,\,(4.38)]{Ioo18b}, where the term $\det\{\Lambda_1,\Lambda_2\}$
should not appear.}

\section{Quantization of group actions}
\label{groupsec}

In this section, we describe the behavior of Berezin-Toeplitz quantization
described in \cref{BTsec} with respect to the action
of a compact Lie group $G$, and its applications to the representation theory
of $G$ via the symplectic geometry of coadjoint orbits.

\subsection{Kostant-Souriau quantization}
\label{KSsec}

%
%

Let $X$ be a smooth manifold endowed with a Hermitian line
bundle $(L,h^L)$ with Hermitian connection $\nabla^L$,
and let $G$ be a compact
Lie group acting on $L$ over $X$, preserving the Hermitian
metric $h^L$ and the Hermitian connection $\nabla^L$.
We write $\mathfrak{g}:=\text{Lie}(G)$ for the Lie algebra of $G$.
We consider the action of $G$ on the space of smooth sections
$\cinf(X,L)$ defined for all $g\in G$, $s\in\cinf(X,L)$ and $x\in X$ by
\begin{equation}\label{action}
(gs)(x):=g.s(g^{-1}x)\,.
\end{equation}
For any $\xi\in\g$, we write $L_\xi$ for the induced \emph{Lie derivative} acting on $s\in\cinf(X,L)$ by
\begin{equation}\label{Lxidef}
L_\xi\,s:=\frac{d}{dt}\Big|_{t=0}e^{-t\xi}s\,,
\end{equation}
and denote by $\xi^X\in\cinf(X,TX)$
the vector field over $X$ induced by the infinitesimal action of $\xi$,
defined by the formula $\xi^X:=\frac{d}{dt}\big|_{t=0}e^{t\xi}x$.
The following definition is taken from \cite[Def.\,7.5]{BGV04}.

\begin{defi}\label{Kostantmudef}
The \emph{moment} $\mu^L:X\to\g^*$ of the action of
$G$ on $(L,h^L,\nabla^L)$ over $X$ is defined by the following formula,
for all $\xi\in\g$ and $s\in\cinf(X,L)$, 
\begin{equation}\label{Kostantmufla}
\<\mu^L,\xi\>\,s
=\frac{\sqrt{-1}}{2\pi}\left(L_\xi\,s-\nabla^L_{\xi^X}s\right)\,.
\end{equation}
\end{defi}

Formula \cref{Kostantmufla} is called the \emph{Kostant formula}.
One readily checks that the right-hand side of formula
\cref{Kostantmufla} is
$\cinf(X)$-linear, so that the left-hand side defines a function
$\<\mu^L,\xi\>\in\cinf(X,\R)$,
and that the application $\mu^L:X\to\g^*$ thus defined is linear and
equivariant with respect to
the \emph{coadjoint action} of $G$ on $\g^*$, defined
for any $g\in G$ and $\lambda\in\g^*$ by
\begin{equation}\label{coaddef}
\<\Ad_g^*\lambda,\xi\>:=\<\lambda,\Ad_{g^{-1}}\xi\>\,.
\end{equation}

In case $(X,\om)$ is a symplectic manifold prequantized by
$(L,h^L,\nabla^L)$ in the sense of \cref{preq}, one readily checks
that \cref{Kostantmudef} defines a \emph{moment map}
$\mu:=\mu^L:X\to\mathfrak{g}^*$ for the symplectic action of $G$ on $(X,\om)$,
so that
\begin{equation}\label{momentfla}
d\<\mu,\xi\>=\iota_{\xi^X}\om\,,
\end{equation}
where $\iota_v$ denotes the interior product with respect to $v\in\cinf(X,TX)$.
Note now that for any $p\in\N$, \cref{Kostantmudef} applied to
the induced action of $G$ on $L^p$ and to $s^p:=s^{\otimes p}\in\cinf(X,L^p)$
readily gives that
$\mu^{L^p}=p\,\mu$, so that for any $s_p\in\cinf(X,L^p)$
and $\xi\in\mathfrak{g}$, we get
\begin{equation}\label{KostantLp}
\frac{\sqrt{-1}}{2\pi}L_\xi s_p=\big(p\,\<\mu,\xi\>+
\frac{\sqrt{-1}}{2\pi}\nabla^{L^p}_{\xi^X}\big)s_p\,.
\end{equation}
More generally, if $G$ acts on a complex line bundle $K$ preserving
a Hermitian metric $h^K$ and a Hermitian connection $\nabla^K$,
\cref{Kostantmudef}
applied $L^p\otimes K$ gives
$\mu^{L^p\otimes K}=p \mu+\mu^K$  for any $p\in\N$,
so that for any $s_p\in\cinf(X,L^p\otimes K)$,
the Kostant formula \cref{Kostantmufla} gives
\begin{equation}\label{KostantK}
\frac{\sqrt{-1}}{2\pi}L_\xi s_p=\big(p\,\<\mu,\xi\>+\<\mu^K,\xi\>+
\frac{\sqrt{-1}}{2\pi}\nabla^p_{\xi^X}\big)s_p\,.
\end{equation}

Let us now assume that $(X,\om)$ is equipped with a compatible $G$-invariant
complex structure $J\in\End(TX)$, making $(X,\om,J)$ into a Kähler manifold
and $(L,h^L)$, $(K,h^K)$ into holomorphic Hermitian line bundles on which
the action
of $G$ lifts holomorphically. In particular, the action of $G$ on
sections preserves the finite-dimensional
subspace $H^0(X,L^p\otimes K)\subset\cinf(X,L^p\otimes K)$ of holomorphic sections,
making $H^0(X,L^p\otimes K)$ into a finite-dimensional representation of $G$,
and the Lie derivative $L_\xi$ defined as in \cref{Lxidef}
preserves $H^0(X,L^p\otimes K)$.
Recalling the \cref{BTdef} of Berezin-Toeplitz quantization,
we then have the following classical formula of Tuynman.

\begin{prop}\label{LxiTp}
{\emph{\cite{Tuy87}}}
For any $\xi\in\mathfrak{g}$ and any $s\in H^0(X,L^p\otimes K)$,
we have
\begin{equation}\label{LxiTp&fla}
\frac{\sqrt{-1}}{2\pi p}L_\xi s=T_p(\<\mu,\xi\>)s
+\frac{1}{p} T_p\big(\<\mu^K,\xi\>+
\frac{1}{4\pi}\Delta^X \<\mu,\xi\>\big)s\,,
\end{equation}
where $\Delta^X$ is the Laplace-Beltrami operator of $(X,g^{TX})$
acting on $\cinf(X,\R)$.
\end{prop}
\begin{proof}
For any $s_1,\,s_2\in H^0(X,L^p\otimes K)$, the Kostant formula \cref{KostantK}
implies
\begin{equation}\label{LxiTpdem}
\frac{\sqrt{-1}}{2\pi}\<L_\xi s_1,s_2\>_{L^2}=
\<(p\,\<\mu,\xi\>+\<\mu^K,\xi\>)s_1,s_2\>_{L^2}+
\frac{\sqrt{-1}}{2\pi}\<\nabla^p_{\xi^X} s_1,s_2\>_{L^2}\,.
\end{equation}
Recall on the other hand that
by definition of a holomorphic section, we have
$\nabla^p_{v^{(0,1)}}s=0$ for all $v\in\cinf(X,TX)$ and
$s\in H^0(X,L^p\otimes K)$.
Thus by definition \cref{L2} of the $L^2$-product and using that the connection
$\nabla^p$ induced by $\nabla^L,\,\nabla^K$ on $L^p\otimes K$
preserves the Hermitian product $h^p$ induced by $h^L,\,h^K$,
we get
\begin{equation}
\begin{split}
\<\nabla^p_{v} s_1,s_2\>_{L^2}&=
\int_X\,h^p(\nabla^p_{v^{(1,0)}} s_1,s_2)\,dv_X\\
&=\int_X L_{v^{(1,0)}}h^p(s_1,s_2)\,dv_X\\
&=-\int_X h^p(s_1,s_2)\,d\iota_{v^{(1,0)}}\frac{\om^n}{n!}\,,
\end{split}
\end{equation}
where we also used Stokes theorem, the definition \cref{voldef} of the Liouville volume form and Cartan's formula $L_\xi=\iota_\xi d+d\iota_\xi$.
Using now the fact that $v^{(1,0)}=v-iJv$, which readily follows from the
definition of the splitting \cref{splitc}, and that $\om$ is $J$-invariant,
we get for any $w\in\cinf(X,TX)$ that
\begin{equation}
\iota_{v^{(1,0)}}\om(w)=\om(v,w)+i\om(v,Jw)=\iota_v\om(w^{(0,1)})\,.
\end{equation}
By the fundamental property \cref{momentfla} of a moment map
and using that $\om$ is closed,
we then compute for all $\xi\in\g$ that
\begin{equation}
\begin{split}
d\iota_{(\xi^{X})^{(1,0)}}\frac{\om^n}{n!}
&=d\,\left(\frac{\dbar\<\mu,\xi\>\wedge\om^{n-1}}{(n-1)!}\right)\\
&=-\frac{\sqrt{-1}}{2}\Delta^X\<\mu,\xi\>\ \frac{\om^{n}}{n!}\,,
\end{split}
\end{equation}
where we also used the
classical formula of Kähler geometry, which can be deduced for instance from
\cite[Ex.\,1.8]{Sze14},
\begin{equation}
\frac{1}{2}\Delta^Xf\ \frac{\om^n}{n!}=
\sqrt{-1}\frac{\dbar\partial f\wedge\om^{n-1}}{(n-1)!}\,.
\end{equation}
Using that the orthogonal projection \cref{projdef}
with respect to the $L^2$-product \cref{L2} restricts to the identity
on $H^0(X,L^p\otimes K)$, we thus get from equation \cref{LxiTpdem} that
\begin{equation}
\begin{split}
\frac{\sqrt{-1}}{2\pi}\<L_\xi s_1,s_2\>_{L^2}&=
\left\langle\big(p\,\<\mu,\xi\>+\<\mu^K,\xi\>+\frac{1}{4\pi}\Delta^X\<\mu,\xi\>\big)s_1,s_2\right\rangle_{L^2}\\
&=\left\langle T_p\big(p\,\<\mu,\xi\>+\<\mu^K,\xi\>+\frac{1}{4\pi}\Delta^X\<\mu,\xi\>\big)s_1,s_2\right\rangle_{L^2}\,,
\end{split}
\end{equation}
which is true for all $s_1,\,s_2\in H^0(X,L^p\otimes K)$. This proves the result.
\end{proof}

Recall now the natural ring isomorphism
\begin{equation}\label{Rg=Sg}
\begin{split}
S(\g)&\xrightarrow{~\sim~}\R[\g^*]\\
\xi_1\cdots\xi_j&\longmapsto
\left[\alpha\mapsto\<\alpha,\xi_1\>\cdots\<\alpha,\xi_j\>\right]
\end{split}
\end{equation}
between the real symmetric algebra $S(\g)$ of $\g$ and
the ring $\R[\g^*]$ of polynomials with real coefficients
over $\g^*$. We then have the following definition, which will play a central
role for the proof of \cref{mainth} in \cref{GZquantsec}.

\begin{defi}\label{Qpdef}
For any $p\in\N$, the \emph{symmetric quantization} of the action of $G$
on $H^0(X,L^p\otimes K)$ is the linear map defined through the identification
\cref{Rg=Sg} by
\begin{equation}\label{Qpfla}
\begin{split}
Q_p:\R[\g^*]&\longrightarrow \End(H^0(X,L^p\otimes K))\\
\xi_1\cdots\,\xi_j &\longmapsto \frac{(-1)^j}{j!}
\sum_{\sigma\in\Sigma_j}L_{\xi_{\sigma(1)}}\cdots L_{\xi_{\sigma(j)}}\,.
\end{split}
\end{equation}
\end{defi}

Note that by definition of the action \cref{action} and of the Lie derivative
\cref{Lxidef}, the operator $Q_p(\xi)\in\End(H^0(X,L^p\otimes K))$ coincides
with the infinitesimal action of $\xi\in\g$ on $H^0(X,L^p\otimes K)$.
The following fundamental semiclassical property of the symmetric
quantization of \cref{Qpdef} readily follows from the tools of Berezin-Toeplitz
quantization described in \cref{BTsec}.

\begin{prop}\label{L=Tp}
For any homogeneous polynomial $P\in\R[\g^*]$ of degree $j\in\N$, we have
the following asymptotic expansion in the sense of the operator norm
as $p\to+\infty$,
\begin{equation}
\frac{1}{(2\pi\sqrt{-1} p)^j}Q_p(P)=T_p(\mu^*P)+O(p^{-1})\,.
\end{equation}
\end{prop}
\begin{proof}
Using \cref{LxiTp} and applying repeatedly formula \cref{BMS1} of \cref{BMS},
we readily get the following asymptotic formula in the sense of the operator
norm as $p\to+\infty$, for any
$\xi_1,\dots,\xi_k\in\g$,
\begin{equation}
\frac{\sqrt{-1}^k}{(2\pi p)^k}L_{\xi_1}\cdots L_{\xi_2}=T_p(\<\mu,\xi_1\>
\cdots\<\mu,\xi_n\>)+O(p^{-1})\,,
\end{equation}
where we also used formula \cref{BMS0} of \cref{BMS} on the terms of
order at most $p^{-1}$. As the right-hand side
does not depend on the ordering of $\xi_1,\dots,\xi_k\in\g$ in the left-hand side,
this implies the result by \cref{Qpdef} via the identification \cref{Rg=Sg}.
\end{proof}

\subsection{Structure and representation theory of compact Lie groups}
\label{structsec}

Let $G$ be a compact Lie group, and fix a \emph{maximal torus} $T\subset G$,
that is an abelian subgroup of $G$ which is maximal for the inclusion among
abelian subgroups. Recall for instance from \cite[Chap.\,IV,\,(1.6)]{BD85}
that any two maximal tori in $G$ are conjugate to each other, and that every
$g\in G$ is contained in a maximal torus.
Write $N(T)\subset G$ for the normalizer of $T$ inside $G$, and
note in particular that $T\subset N(T)$. 
Following \cite[Chap.\,IV]{BD85}, let us write
\begin{equation}\label{Wdef}
W:=N(T)/T
\end{equation}
for the \emph{Weyl group} of $G$, which is a finite group with an action
on $T$ by automorphisms induced from the action of $N(T)$ by conjugation, so
that two elements of $T$ are conjugate in $G$ if and only if they belong to
the same orbit of $W$.

Write $\g:=\Lie(G)$ for the Lie algebra of $G$ and $\t:=\Lie(T)\subset\g$
for the Lie subalgebra of $T\subset G$. Let us fix once and for all a
$G$-invariant scalar product $\<\cdot,\cdot\>_\g$ on $\g$, inducing
identifications $\g\simeq\g^*$ and $\t\simeq\t^*$ which will be understood
in the whole sequel.
In particular, we will often consider the inclusion $\t^*\subset\g^*$
induced by $\t\subset\g$.
We write $\<\cdot,\cdot\>_{\g^*}$ for the scalar product on
$\g^*$ induced by $\<\cdot,\cdot\>_\g$.
Following \cite[Chap.V,\,(2.15)]{BD85}, let $\Gamma^*\subset\t^*$ be the
\emph{lattice of integral forms}, defined by
\begin{equation}\label{LT}
\Gamma^*:=\{\alpha\in\t^*~|~\<\alpha,\xi\>\in \Z,
\text{ for all }\xi\in\t\text{ such that }\exp(\xi)=1\}\,.
\end{equation}
Following \cite[Chap.V,\,\S2]{BD85},
there exists a finite subset $R\subset \Gamma^*$, called the
\emph{set of roots}, such that we have a decomposition
\begin{equation}\label{gCdec}
\g_\C=\t_\C\oplus\bigoplus_{\alpha\in R}(\g_\C)_\alpha\,,
\end{equation}
of the complexified Lie algebra $\g_\C:=\g\otimes\C$ into the common
eigenspaces
\begin{equation}\label{gCalpha}
(\g_\C)_\alpha:=
\{\xi\in\g~|~
[\eta,\xi]=2\pi\sqrt{-1}\<\alpha,\eta\>\,\xi\text{ for all }\eta\in\t\}\,,
\end{equation}
for the adjoint action of $\t$ on $\g$, which satisfy $\dim_\C\g_\alpha=1$
for all $\alpha\in R$ and where $\t_\C:=\t\otimes\C$.
Let us define the open set of
\emph{regular points} of $\t^*$ by
\begin{equation}\label{tregdef}
\t^*_{\text{reg}}:=\{\lambda\in\t^*~|~\<\alpha,\lambda\>_{\g^*}\neq 0\,
\text{ for all }\alpha\in R\}\,.
\end{equation}
The connected components of $\t^*_{\text{reg}}\subset\t^*$ are called
\emph{Weyl chambers}, and the Weyl group
$W$ acts simply and transitively on the set of Weyl chambers via the action
on $\t^*$ induced by its action on $T$. We
fix once and for all a set of \emph{positive roots} $R_+\subset R$, which by
definition satisfies
\begin{equation}\label{posroot}
R=R_+\cup(-R_+)\ \text{ and }\ R_+\cap(-R_+)=\0\,.
\end{equation}
This allows to single out a Weyl chamber, called the fundamental or
\emph{positive Weyl chamber} corresponding to $R_+\subset R$, defined by
\begin{equation}\label{Weylchamber}
\t^*_+:=\{\lambda\in\t^*~|~\<\alpha,\lambda\>_{\g^*}>0
\ \text{ for all }\ \alpha\in R_+\}\,.
\end{equation}
Following \cite[Chap.V,\,(4.11)]{BD85}, let us finally introduce
the \emph{half-sum of positive roots}
\begin{equation}\label{rhodef}
\rho_{\g}:=\frac{1}{2}\sum_{\alpha\in R_+}\alpha\in\t^*\,,
\end{equation}
which plays a key role in the structure theory of compact Lie groups, as well
as their representation theory described below.

Let now $V$ be a finite-dimensional unitary representation of $G$,
let $T_0\subset T$ be a subtorus of the maximal torus of $G$,
and write $\t_0:=\Lie(T_0)\subset\t$ for its Lie subalgebra.
Following \cite[Chap.\,III,\,\S\,2]{Kna01},
the action of $T_0\subset G$ induced on $V$
by restriction
leads to a unique orthogonal decomposition
\begin{equation}\label{Tirrepdec}
V=\bigoplus_{\lambda\in\Gamma^*\cap\,\t_0^*}\,V_\lambda\,,
\end{equation}
preserved by the action of $T_0$, where
\begin{equation}\label{wghtspacedef}
V_\lambda
:=\{v\in V~|~\exp(\xi).v=e^{2\pi\sqrt{-1}\<\lambda,\xi\>}v\text{ for all }\xi\in\t_0\,\}\,.
\end{equation}
The finite subset
\begin{equation}\label{WVdef}
\Gamma_V^{T_0}:=\{\lambda\in\Gamma^*\cap\t_0~|~V_\lambda\neq 0\}\subset\Gamma^*
\end{equation}
is called the \emph{set of weights} of the action of $T_0$ on $V$.
If $T_0=T$, we simply write $\Gamma_V:=\Gamma_V^{T_0}$, and call it
the \emph{set of weights} of $V$.
Note that it is independent of the choice of the maximal torus
$T\subset G$ since two maximal tori are conjugate, and that the
finite-dimensional representation $\g_\C$ of $G$ by adjoint action has
$\Gamma_{\g_\C}=R\cup\{0\}$ as set of weights.

Following for instance \cite[Th.\,4.28]{Kna01}, we introduce
the following notion, which plays a fundamental role in representation theory.
\begin{defi}
\label{highestwghtdef}
An irreducible representation $V$ of $G$
is said to have \emph{highest weight} $\lambda\in \Gamma_V$ if
for all $\eta\in \Gamma_V$, there exist $n_\alpha\geq 0$ for all $\alpha\in R_+$ such that
\begin{equation}
\lambda-\eta=\sum_{\alpha\in R_+}\,n_\alpha\,\alpha\,.
\end{equation}
\end{defi}
Write $\overline{\mathfrak{t}^*_+}\subset\t^*$ for the closure of
the positive Weyl chamber \cref{Weylchamber} inside $\t^*$.
The following fundamental
classification of finite-dimensional representations of $G$ is due to Weyl.

\begin{theorem}\label{highestwghtth}
{\cite[Th.\,4.28]{Kna01}}
There is a bijective correspondence between the
set of unitary irreducible representations of $G$ and the discrete set
\begin{equation}\label{Gammadef}
\overline{\Gamma}_+^*:=\Gamma^*\cap\overline{\mathfrak{t}^*_+}\,,
\end{equation}
in such a way that for any $\lambda\in\Gamma_+^*$,
there exists a unique finite-dimensional irreducible representation
$V(\lambda)$ of $G$ with highest weight $\lambda$.
\end{theorem}
The discrete set \cref{Gammadef} parameterizing the irreducible
representations of $G$ is called the
\emph{lattice of integral dominant weights}.

Let us finish this Section by describing the case of
the unitary group $G=U(n)$ for any $n\in\N^*$,
with Lie algebra $\u(n):=\textup{Lie}(U(n))$, which we adapt from
\cite[p.\,120]{GS83} to fit with our conventions.
Following \cite[Chap.\,IV,\,\S\,3]{BD85},
we make the natural choice of maximal torus
$T\subset U(n)$ consisting of the diagonal matrices, so that the associated
Weyl group \cref{Wdef} coincides with the $n^{\text{th}}$ symmetric group
$\mathfrak{S}_n$ acting on
diagonal matrices in $T\subset G$ by permutation of the diagonal entries.
Recall that $\u(n)$ is naturally identified with the space of $n\times n$
anti-Hermitian matrices,
and consider the $U(n)$-equivariant identification of its dual $\u(n)^*$
with the space
$\Herm(\C^n)$ of $n\times n$ Hermitian matrices given by
\begin{equation}\label{Herm=un}
\begin{split}
\Herm(\C^n)&\xrightarrow{~\sim~}\u(n)^*\\
2\pi\sqrt{-1}A&\longmapsto\big[B\mapsto\Tr[AB]\big]\,.
\end{split}
\end{equation}
We endow $\u(n)^*$ with the $U(n)$-invariant scalar product
$\<\cdot,\cdot\>_{\g^*}$
induced by the trace product on $\Herm(\C^n)$
under this identification. Then the subspace
$\t^*\subset\g^*$ induced by $\t\subset\g$
coincides  under this identification with the subspace of $\Herm(\C^n)$
consisting of diagonal matrices with real coefficients,
so that we have a natural identification
\begin{equation}\label{t*=Rn}
\t^*\simeq\R^n\,,
\end{equation}
on which the natural action of the Weyl group $W=\mathfrak{S}_n$ acts by
permutations of the coordinates. Then
the lattice of integral weights \cref{Gammadef} is given
by the lattice of integral points $\Gamma^*\simeq\Z^n$. Following
\cite[Chap.V,\,\S\,6]{BD85}, there exists a natural
choice of positive roots \cref{posroot} for which
the associated positive Weyl chamber \cref{Weylchamber} is given by
\begin{equation}\label{WeylchamberUn}
\mathfrak{t}^*_+=\{(\lambda_1,\dots,\lambda_n)\in\R^n
~|~\lambda_1>\lambda_2>\cdots>\lambda_n\}\,,
\end{equation}
so that the lattice of dominant weights  \cref{Gammadef} is given by
\begin{equation}\label{GammaUndef}
\overline{\Gamma}^*_+=\{(\lambda_1,\dots,\lambda_n)\in\Z^n
~|~\lambda_1\geq\lambda_2\geq\cdots\geq\lambda_n\}\,.
\end{equation}
By \cite[Chap.V,\,Prop.\,6.2.(vi)]{BD85},
the half-sum of positive roots \cref{rhodef}
is given by
\begin{equation}\label{rhon}
\rho_{\u(n)}=
\left(\frac{n-1}{2},\frac{n-3}{2},\dots,-\frac{n-3}{2},-\frac{n-1}{2}\right)\,.
\end{equation}

\subsection{Coadjoint orbits}
\label{coadsec}



Recall that we fixed a
$G$-invariant scalar product $\<\cdot,\cdot\>_\g$ on $\g$, inducing an
identification $\g\simeq\g^*$. Via this identification,
the usual adjoint action of $G$ on $\g$ corresponds to
the coadjoint action \cref{coaddef} of $G$ on $\g^*$.
Recall also that we chose a maximal torus $T\subset G$,
and that we have an inclusion $\t^*\subset\g^*$ induced by the $G$-invariant
scalar product on $\g$. For any $\lambda\in\mathfrak{t}^*$, we write
$X_\lambda\subset\mathfrak{g}^*$ for the associated
\emph{coadjoint orbit}, defined by
\begin{equation}\label{Oldef}
X_\lambda:=\{\Ad_g^*\lambda\in\g^*~|~
g\in G\}\,.
\end{equation}
Note that since any element in $G$ is conjugate to an element in $T\subset G$,
all coadjoint orbits in $\g^*$ are of the form $X_\lambda\subset\g^*$
for some $\lambda\in\t^*$, and that two elements $\lambda,\,\eta\in\t^*$
belong to the same coadjoint orbit $X_\lambda=X_\eta\subset\g^*$ if
and only if they belong to the same orbit of the Weyl group \cref{Wdef}.

Note also that the coadjoint action of $G$ on $X_\lambda$ is transitive by
definition, so that its tangent space $T_\alpha X_\lambda$
at any point $\alpha\in X_\lambda$ is spanned by 
the vector fields
$\xi^{X_\lambda}\in\cinf(X_\lambda,TX_\lambda)$ 
induced by the infinitesimal action of $\g$, for all $\xi\in\g$.
The following fundamental
property of coadjoint orbits can be found in \cite[Lem.\,7.22]{BGV04}.

\begin{prop}\label{KKSprop}
For any $\lambda\in\t^*$, the associated coadjoint orbit $X_\lambda$
is endowed with a natural $G$-invariant symplectic form
$\om\in\Om^2(X_\lambda,\R)$, defined for any
$\xi_1,\,\xi_2\in\mathfrak{g}$ and $\alpha\in X_\lambda$ by
\begin{equation}\label{KKS}
\om_\alpha(\xi_1^{X_\lambda},\xi_2^{X_\lambda})=\alpha([\xi_1,\xi_2])\,,
\end{equation}
and the action of $G$ on $(X_\lambda,\om)$ is Hamiltonian,
with moment map
\begin{equation}\label{momentdef}
\mu:X_\lambda\lhook\joinrel\longrightarrow\mathfrak{g}^*
\end{equation}
given by the inclusion.
\end{prop}

The symplectic form \cref{KKS} is called the 
\emph{Kirillov-Kostant-Souriau} symplectic form over the coadjoint
orbit $X_\lambda$.

Let us write $G_\lambda\subset G$
for the stabilizer of $\lambda\in\t^*\subset\g^*$
by the coadjoint action \cref{coaddef}, so that its Lie algebra
$\g_\lambda:=\textup{Lie}(G_\lambda)$ is given by
\begin{equation}\label{glambdadef}
\g_\lambda=\big\{\xi\in\g~\big|~\frac{d}{dt}\Big|_{t=0}\Ad^*_{e^{t\xi}}\lambda=0\big\}\,.
\end{equation}
The transitive action of $G$ on
$X_\lambda$ then induces a natural identification
\begin{equation}\label{Ol=GGl}
\begin{split}
G/G_\lambda&\xrightarrow{~\sim~} X_{\lambda}\\
[g]&\longmapsto\Ad_g^*\lambda\,.
\end{split}
\end{equation}
Since the coadjoint action \cref{coaddef} corresponds
to the adjoint action via the identification $\g\simeq\g^*$, we see that
the coadjoint action of $T\subset G$ on $\g^*$ restricts to a trivial action
on $\t^*\subset\g^*$, so that formula \cref{glambdadef} implies the
inclusion $T\subset G_\lambda$ as a maximal torus of $G_\lambda$.
Furthermore, it readily follows from the definitions \cref{gCalpha}
of the
roots and \cref{tregdef} of the regular points that
\begin{equation}\label{T=Glambda}
T=G_\lambda\,\text{ if and only if }\,\lambda\in\t^*_{\text{reg}}\,.
\end{equation}
In this case, the coadjoint orbit $X_\lambda$ is called a
\emph{regular coadjoint orbit}, and via the identification
\cref{Ol=GGl}, we see that
regular coadjoint orbits are the orbits of maximal dimension among coadjoint
orbits in $\g^*$.

Let us now assume that $\lambda\in\Gamma^*\subset\t^*$ belongs
to the lattice of integral forms \cref{LT}.
Then following \cite[Th.\,5.7.1]{Kos70}, the induced character
\begin{equation}\label{sigmal}
\begin{split}
\sigma_\lambda:T&\longrightarrow S^1\subset\C^*\\
\xi&\longmapsto e^{2\pi\sqrt{-1}\<\lambda,\xi\>}\,,
\end{split}
\end{equation}
extends uniquely to a character
$\sigma_\lambda:G_\lambda\to S^1\subset\C^*$ of $G_\lambda$,
and the $G$-equivariant line bundle
\begin{equation}\label{Lgamdef}
L_\lambda:=G\times_{\sigma_\lambda}\C\longrightarrow G/G_\lambda
\end{equation}
is naturally endowed with a $G$-equivariant
connection $\nabla^{L_\lambda}$ preserving
the Hermitian metric $h^{L_\lambda}$ induced by the
standard Hermitian metric of $\C$.
Via the identification \cref{Ol=GGl}, we then get a natural
Hermitian line bundle with connection
$(L_\lambda,h^{L_\lambda},\nabla^{L_\lambda})$ over the coadjoint orbit
$X_\lambda$ on which the $G$-action lifts, preserving Hermitian metric
and connection.
The following classical result is due to Kostant, and can be found in
\cite[Prop.\,8.6]{BGV04}.

\begin{prop}\label{Lpreqprop}
{\cite[Th.\,5.7.1,\,Cor.\,1]{Kos70}}
For any $\lambda\in\Gamma^*$, the Hermitian connection
$\nabla^{L_\lambda}$ on the Hermitian line bundle line
$(L_\lambda,h^{L_\lambda})$ over the coadjoint orbit $X_\lambda$
prequantizes the symplectic form \cref{KKS} in the sense
of \cref{preq}, and the moment map
\cref{momentdef} satisfies the Kostant formula \cref{Kostantmufla}
for the action of $G$ on $(L_\lambda,h^{L_\lambda},\nabla^{L_\lambda})$.
\end{prop}

Recall now the choice \cref{posroot} of
a set of positive roots $R_+\subset R$.
For any $\lambda\in\t^*\subset\g^*$, recall that
$T\subset G$ is included as a maximal torus in the stabilizer
$G_\lambda\subset G$,
and write $R_\lambda\subset\t^*$ for the set of roots of
$G_\lambda$. By definition \cref{gCalpha} of the set of roots
and by definition \cref{glambdadef} of the Lie algebra
$\g_\lambda=\text{Lie}(G_\lambda)$, one
readily checks that we have $R_\lambda\subset R$, and that
$R_\lambda^+:=R_\lambda\cap R^+$ defines a set of positive roots of
$R_\lambda$.
In particular,in the decomposition
\cref{gCdec} of $\g_\C$, the complexified subalgebra
$\g_\lambda\otimes\C\subset\g_\C$ coincides with
\begin{equation}
\g_\lambda\otimes\C=\t_\C\oplus\bigoplus_{\alpha\in R_\lambda}(\g_\C)_\alpha\,,
\end{equation}
which induces an identification
\begin{equation}
T_\C\,(G/G_\lambda)=\bigoplus_{\alpha\in R\backslash R_\lambda}
(\g_\C)_\alpha\,.
\end{equation}
Following \cite[p.\,257]{BGV04}, we then get a natural $G$-invariant complex structure on $G/G_\lambda$ defined through the decomposition
\cref{splitc} by
\begin{equation}\label{T10O}
T^{(1,0)}(G/G_\lambda)=\bigoplus_{\alpha\in R^+\backslash R_\lambda^+}(\g_\C)_\alpha\,.
\end{equation}
Via the identification \cref{Ol=GGl}, this induces a $G$-invariant
complex structure on $X_\lambda$ compatible with
the natural symplectic form of \cref{KKSprop}.
As explained in \cref{BTsubsec},
in the case $\lambda\in\Gamma^*$, we get from \cref{Lpreqprop} a
natural holomorphic structure on the complex line bundle \cref{Lgamdef} 
over $X_\lambda$ induced by $\nabla^{L_\lambda}$.
Note that if $\lambda\in\t^*_{\text{reg}}$,
so that $G_\lambda=T$ by
\cref{T=Glambda}, then for any $\eta\in\t^*_{\text{reg}}$,
the identification \cref{Ol=GGl} induces an identification
$X_\lambda\simeq X_\eta$, and this identification is biholomorphic
if and only if $\lambda$ and $\eta$ belong to the same Weyl chamber.

Let now $\lambda\in\t^*_{\text{reg}}$,
and recall the definition \cref{rhodef} of the half-sum of positive roots $\rho_\g\in\t^*$.
Then by definition \cref{KX} of the canonical
line bundle $K_{X_\lambda}$ of $X_\lambda$ associated with the complex
structure \cref{T10O} and the fact that $\dim_\C\g_\alpha=1$
for all $\alpha\in R$,
we readily compute that
\begin{equation}\label{cancoad}
(K_{X_\lambda},h^{K_{X_\lambda}})=(L_{-2\rho_\g},h_{-2\rho_\g})
\end{equation} 
as Hermitian holomorphic line bundles over $X_\lambda\simeq G/T$
via the identification \cref{Ol=GGl}.
In particular, in case the half-sum of roots \cref{rhodef} satisfies
$\rho\in\Gamma^*$, the line bundle $L_{-\rho_\g}$ over $G/T$ as in \cref{Lgamdef}
is well defined, and via the identification \cref{Ol=GGl},
we get a canonical
choice of a square root of $K_{X_\lambda}$ over $X_\lambda$ given by
\begin{equation}\label{metacoad}
(K_{X_\lambda}^{1/2},h^{K_{X_\lambda}^{1/2}})
=(L_{-\rho_\g},h_{-\rho_\g})\,.
\end{equation}
%

\subsection{Borel-Weil theorem and the orbit method}
\label{repsec}


Let now $\lambda\in\Gamma^*_+\subset\t^*$ belong to the set of
dominant integral weights \cref{Gammadef}. Following
\cref{Lpreqprop}, consider the
holomorphic Hermitian line bundle
$(L_\lambda,h^{L_\lambda})$ with Hermitian connection $\nabla^{L_\lambda}$
prequantizing $(X_\lambda,\om)$ endowed with the compatible
complex structure induced by \cref{T10O}.
The natural action of $G$ induced by \cref{Lgamdef}
is holomorphic,
so that one can consider the unitary representation
of $G$ given by the associated
space of holomorphic sections $H^0(X_\lambda,L_\lambda)$ endowed with the
$L^2$-product \cref{L2}.
We then have the following version due to Kostant
of a
famous result due to Borel-Weil and Bott, in a form that
can be found for instance in \cite[Th.\,3.7]{GS82} and which constitutes
the building block of the so-called \emph{orbit method}
in representation theory.

\begin{theorem}\label{BWth}
{\cite[\S\,4.3]{Akh96}}
For any $\lambda\in\overline{\Gamma}^*_+$,
we have the following identity of unitary representations of $G$,
\begin{equation}
V(\lambda)=(H^0(X_\lambda,L_\lambda),\<\cdot,\cdot\>_{L^2})\,,
\end{equation}
where $V(\lambda)$ is the irreducible representation of $G$ with highest
weight $\lambda$.
\end{theorem}

Note that if $\lambda\in\Gamma_+^*$ belongs to the set of
\emph{regular dominant weights}
\begin{equation}\label{Gamma0def}
\Gamma^*_+:=\Gamma^*\cap\t_+^*\,,
\end{equation}
so that we have $X_\lambda\simeq G/T$ by \cref{T=Glambda},
then for any other $\eta\in\Gamma_+^*$, as it belongs to the same
Weyl chamber \cref{Weylchamber}, the $G$-equivariant identification
$X_\lambda\simeq X_\eta$ induced by \cref{Ol=GGl} is biholomorphic,
so that \cref{Lgamdef} defines a holomorphic
Hermitian line bundle with connection
$(L_\eta,h^{L_\eta},\nabla^{L_\eta})$ over $(X_\lambda,\om)$, and
by \cref{BWth} we get
\begin{equation}\label{BWcor}
V(\eta)=(H^0(X_\lambda,L_\eta),\<\cdot,\cdot\>_{L^2})\,,
\end{equation}
where $V(\eta)$ is the irreducible representation of $G$ with highest
weight $\eta\in\Gamma^*_+$. Note also from \cref{sigmal} that for any
$\eta_1,\,\eta_2\in\Gamma^*_+$, we have
$L_{\eta_1+\eta_2}=L_{\eta_1}\otimes L_{\eta_2}$.
Furthermore, for any $\lambda\in\overline{\Gamma}^*_+$ and $p\in\N$,
we get from the
same argument the following consequence of \cref{BWth},
\begin{equation}\label{BWcor2}
V(p\lambda)=(H^0(X_\lambda,L_\lambda^p),\<\cdot,\cdot\>_{L^2})\,.
\end{equation}

Following now \cite[\S\,17.2]{Hum78}, let us consider
the \emph{universal enveloping algebra} $U(\g)$ of $\g$,
with associated Lie algebra morphism
\begin{equation}\label{idef}
j:\g\longrightarrow U(\g)\,,
\end{equation}
which
is characterized by the universal property that for any vector space $V$ and
any Lie algebra morphism
$\varphi:\g\to\End(V)$, there exists a unique algebra morphism
\begin{equation}\label{Phi}
\Phi:U(\g)\longrightarrow\End(V)
\end{equation}
such that $\varphi=\Phi\circ j$. In particular,
for any representation $V$ of $G$, one gets an induced representation
of $U(\g)$ on $V$ by taking $\varphi:\g\to\End(V)$ to be the associated
infinitesimal action.
Under the identification \cref{Rg=Sg} of the ring of polynomials $\R[\g^*]$
with real coefficients over $\g^*$ with the symmetric algebra of $\g$, we can
then state the celebrated \emph{Poincaré-Birkhoff-Witt theorem}.

\begin{theorem}\label{PBWprop}
{\cite[Prop.\,2.4.10]{Dix96}}
The \emph{symmetrization map}
\begin{equation}\label{sym}
\begin{split}
\text{sym}:\R[\g^*]&\xrightarrow{~\sim~} U(\g)\\
\xi_1\cdots \xi_k &\longmapsto\frac{1}{k!}
\sum_{\sigma\in\mathfrak{S}_k}
j(\xi_{\sigma(1)})\cdots\,j(\xi_{\sigma(k)})\,,
\end{split}
\end{equation}
is a bijective $\g$-equivariant linear map for the natural adjoint
action on both spaces.
\end{theorem}

Note that \cref{PBWprop} implies in particular that the map
$j:\g\to U(\g)$ of \cref{idef}
is injective, and we will call it as the \emph{inclusion map} of $\g$ into
$U(\g)$.
\cref{PBWprop} also induces an adjoint action of $G$
on $U(\g)$ by exponentiating the adjoint action of $\g$ induced by the
inclusion map \cref{idef}, and for which the symmetrization map \cref{sym}
is $G$-equivariant.
Note also that the symmetrization map \cref{sym} is not an
isomorphism of algebras in general, as the first space is commutative, while
the second is not commutative in general since the inclusion map
\cref{idef} is a Lie algebra
morphism. Finally, in the setting of \cref{KSsec}, if one
considers the representation of $G$ given by $V=H^0(X,L^p\otimes K)$, then
by \cref{Qpdef} and definition \cref{sym} of the symmetrization map, for any
$P\in\R[\g^*]$ one gets
\begin{equation}\label{Qp=Phi}
Q_p(P)=\Phi(\sym(P))\,,
\end{equation}
where $\Phi:U(\g)\to\End(H^0(X,L^p\otimes K))$ is the induced action
\cref{Phi} of $U(\g)$ on $H^0(X,L^p\otimes K)$.


Let us write $Z[U(\g)]\subset U(\g)$ for the center of the universal
enveloping algebra. Following for instance \cite[\S\,6.1]{Hum78},
the usual Schur lemma readily implies that $Z[U(\g)]$
acts by scalar multiplication on any irreducible
representation of $G$. Recall on the other hand the action of
the Weyl group $W$ on $\t^*$  defined by \cref{Wdef}, and denote by 
$\R[\t^*]^W$ the algebra of $W$-invariant polynomials with real coefficients over $\t^*$.
We then have the following
result due to Harish-Chandra, which can be found in
\cite[Prop.\,7.4.4,\,Th.\,7.4.5]{Dix96}.


\begin{theorem}\label{HCth}
{\cite{H51}}
There is an isomorphism
\begin{equation}\label{HCdef}
\gamma:Z[U(\g)]\xrightarrow{~\sim~} \R[\t^*]^W\,,
\end{equation}
characterised by the fact that for any $z\in Z[U(\g)]$ and
$\lambda\in\overline{\Gamma}_+^*$, we have
\begin{equation}\label{HCfla}
\Phi(z).v =\gamma(z)(2\pi\sqrt{-1}(\lambda+\rho_\g))\,v\,,
\end{equation}
for all $v\in V(\lambda)$, where 
$\Phi:U(\g)\to\End(V(\lambda))$ is the induced action
\cref{Phi} of $U(\g)$ on the irreducible representation $V(\lambda)$
of $G$ with
highest weight $\lambda$, and where $\rho_\g\in\t^*$ denotes
the half-sum of roots \cref{rhodef}.
\end{theorem}

Note that the $2\pi\sqrt{-1}$ comes from our convention \cref{WVdef} for
the weights, and that we are using the fact
that a polynomial over a real vector space
extends to its complexification in the obvious way.
Note also that, since $\frac{1}{p}\Gamma^*_+$
becomes dense inside the open cone $\overline{\t^*_+}\subset\t^*$
as $p\to+\infty$ by \cref{LT} and
\cref{Gammadef}, one readily infers
that a polynomial in $\R[\t^*]$ is uniquely determined by its values
on $\overline{\Gamma}_+^*\subset\t^*$, so that
formula \cref{HCfla} in fact characterizes
$\gamma(z)\in\R[\t^*]^W$ for all $z\in Z[U(\g)]$.

%

%

Since the inclusion map \cref{idef} is injective
and its image generates $U(\g)$ by the Poincaré-Birkhoff-Witt
\cref{PBWprop}, the center $Z[U(\g)]\subset U(\g)$
coincides with the subalgebra of $U(\g)$ of $\g$-invariant
elements by the adjoint action, and
since the symmetrization map \cref{sym} is $\g$-equivariant,
the
Poincaré-Birkhoff-Witt \cref{PBWprop} implies that it
restricts to an isomorphism of vector spaces
\begin{equation}\label{symcenter}
\sym:\R[\g^*]^G \xrightarrow{~\sim~} Z[U(\g)]\,,
\end{equation}
where $\R[\g^*]^G\subset\R[\g^*]$ denotes
the space of real polynomials over $\g^*$ which are
invariant with respect to the coadjoint action \cref{coaddef} of $G$.
Together with the Borel-Weil \cref{BWth},
\cref{L=Tp} then has the following important consequence,
which is typical of a semiclassical result in the orbit method.

\begin{cor}\label{corL=Tp}
For any homogeneous $G$-invariant polynomial $P\in \R[\g^*]^{G}$
of degree $j\in\N$ and any $\lambda\in\t^*$,
we have the following asymptotic estimate as $p\to+\infty$,
\begin{equation}\label{corL=TPfla}
\frac{1}{p^j}\gamma(\sym(P))(p\lambda)
=P(\lambda)+O(p^{-1})
\,,
\end{equation}
so that $\gamma(\sym(P))\in\R[\t^*]^{W}$ has degree $j$ and its
homogeneous part of highest degree coincides with the restriction of
$P\in\R[\g^*]^G$ to $\t^*\subset\g^*$.
\end{cor}
\begin{proof}
Let $\lambda\in\overline{\Gamma}^*_+$,
and recall from \cref{KKSprop} that the moment map
$\mu:X_\lambda\to\g^*$ associated with the coadjoint action of $G$ on
the coadjoint orbit $(X_\lambda,\om)$ is given by the inclusion,
so that since the restriction of a homogeneous $G$-invariant polynomial
$P\in \R[\g^*]^{G}$ to a coadjoint orbit is constant, we have
$\mu^*P\equiv P(\lambda)$. Now \cref{Lpreqprop} states that
the Hermitian line bundle with connection
$(L_\lambda,h^{L_\lambda},\nabla^{L_\lambda})$
prequantizes $(X_\lambda,\om)$, and \cref{BTdef}
for $X=X_\lambda$, $L=L_\lambda$
and $K=\C$ implies that $T_p(\mu^*P)=P(\lambda)\Id$ for all $p\in\N$,
so that \cref{L=Tp} gives the following asymptotic as $p\to+\infty$,
for any sequence $\{s_p\in H^0(X,L^p)\}_{p\in\N}$ with $\|s_p\|_{L^2}=1$,
\begin{equation}\label{corL=TPfla1}
\frac{1}{(2\pi\sqrt{-1} p)^k}Q_p(P)s_p=P(\lambda)s_p+O(p^{-1})\,.
\end{equation}
On the other hand, by \cref{Qp=Phi},
the operator
$Q_p(P)\in\End(H^0(X,L^p))$
corresponds to the action of $\sym(P)\in U(\g)$ on the irreducible representation
of $G$ with highest weight $p\lambda$ via \cref{BWcor2}, so that
by \cref{HCth}, for all $p\in\N$ we get
\begin{equation}\label{corL=TPfla2}
Q_p(P)s_p=
\gamma(\sym(P))(2\pi\sqrt{-1}(p\lambda+\rho_\g))s_p\,.
\end{equation}
As $\gamma(\sym(P))\in\R[\t^*]^W$ is a polynomial, formula
\cref{corL=TPfla} for $\lambda\in\overline{\Gamma}^*_+$ then follows from comparing
the asymptotic expansion
\cref{corL=TPfla1} with the expansion in powers of $p\in\N$ of
the right-hand side of formula \cref{corL=TPfla2}.
As explained after \cref{HCth}, a polynomial in $\R[\t^*]$
is uniquely determined by its values
on $\overline{\Gamma}_+^*\subset\t^*$, so that this concludes the proof.

\end{proof}

Note that \cref{corL=Tp} can also be seen
as a consequence of the explicit description of the
Harish-Chandra isomorphism \cref{HCdef} given for instance
in \cite[\S\,9.4]{NV21}.
%


\section{Quantization of the projective space}
\label{quantprojsec}

In this section, we fix $n\in\N^*$ and consider first the case of irreducible
representations of $G=U(n)$
with highest weight $\lambda=(p,0,\dots,0)\in\overline{\Gamma}^*_+$
for all $p\in\N$ in the sense
of \cref{highestwghtdef} and under the identification \cref{GammaUndef}.
In particular, we establish \cref{toricth} using the
tools of \cref{isosec}.

\subsection{Projective space as a coadjoint orbit}
\label{projsec}

Write $\<\cdot,\cdot\>$ for the natural Hermitian product of $\C^n$
and $|\cdot|$ for the associated norm.
Under the identification \cref{GammaUndef}, let
$\lambda:=(1,0,\dots,0)\in\overline{\Gamma}^*_+\subset\Z^n$, and note
that it corresponds in the identification \cref{Herm=un} to the
Hermitian matrix
$\Pi_{e_1}\in\Herm(\C^n)$
of orthogonal projection on the first vector
$e_1\in\C^n$ in the canonical basis of
$\C^n$. Since any vector $z=(z_1,\dots,z_n)\in\C^n$ of norm $1$ is of the
form $z=ge_1$ for some $g\in U(n)$,
we see that under the identification \cref{Herm=un},
the coadjoint orbit $X_\lambda$ is given by
\begin{equation}\label{XPie1}
X_\lambda=\{\Pi_{z}\in\Herm(\C^n)~|~|z|=1\}\,.
\end{equation}
We thus get a natural $U(n)$-equivariant identification of $X_\lambda$
with the
\emph{complex projective space} $\CP^{n-1}$ of complex lines inside $\C^n$
via the map
\begin{equation}\label{Xl=CPn}
\begin{split}
X_\lambda&\xrightarrow{~\sim~}\CP^{n-1}\\
\Pi_z&\longmapsto[z]\,.
\end{split}
\end{equation}
Using the explicit description of the decomposition \cref{gCdec}
given for instance in \cite[Chap.V,\,\S\,6]{BD85}, one readily checks that
this identification
is biholomorphic for the complex structure induced by \cref{T10O} on
$X_\lambda$. Furthermore, the natural symplectic form \cref{KKS} on
$X_\lambda$ is sent to the canonical Fubini-Study symplectic form of
$\CP^{n-1}$ of volume $1$, and the prequantizing line bundle
$(L_\lambda,h^{L_\lambda},\nabla^{L_\lambda})$ of \cref{Lpreqprop}
corresponds to the dual of the tautological line bundle
with associated Fubini-Study Hermitian metric and connection.
From now on, we will set $X:=X_\lambda$ and $L:=L_\lambda$,
and use freely the identification \cref{Xl=CPn}.

Let now $T_0\subset U(n)$ be the $(n-1)$-dimensional
torus of diagonal matrices with highest-left coefficient equal to $1$,
and write $\t_0:=\text{Lie}(T_0)$ for its Lie algebra.
The associated moment map $\mu_0:X\to\mathfrak{t}^*_0$
by the Kostant
formula \cref{Kostantmufla} satisfies $\mu_0=\mu\circ\pi$,
where $\mu:X\hookrightarrow\u(n)^*$ is the standard moment map
\cref{momentdef} given by the inclusion and $\pi:\u(n)^*\to\mathfrak{t}^*_0$
is the projection induced by the natural inclusion $\t_0\subset\u(n)$.
Consider the identification $\t_0^*\simeq\R^{n-1}$
induced by the identification \cref{t*=Rn}. Using the description
\cref{XPie1} of $X:=X_\lambda$, one
readily checks that
\begin{equation}\label{mu0def}
\begin{split}
\mu_0:X&\longrightarrow\R^{n-1}\\
\Pi_z&\longmapsto(|z_2|^2,|z_3|^2,\dots,|z_{n}|^2)\,.
\end{split}
\end{equation}
In particular, as $\sum_{j=2}^{n}|z_j|^2=1-|z_1|^2$
for all $z\in\C^n$ with $|z|=1$, its image over
$X$ as in \cref{XPie1} is given by the simplex
%
%

\begin{equation}\label{Deltaproj}
\Delta:=\big\{(\nu_2,\dots,\nu_{n})\in\R^{n-1}
~\big|~\sum_{j=2}^{n}\nu_j\leq 1\text{ and }
\nu_k\geq 0\,,\text{ for all }2\leq k\leq n\big\}\,.
\end{equation}
We write
\begin{equation}\label{Lambdabprojdef}
\iota_\nu:\Lambda_\nu:=\mu^{-1}(\nu)\longhookrightarrow X\,,
\end{equation}
for the fibre of \cref{mu0def} over any $\nu\in\Delta$.
As the torus $T_0\subset U(n)$ acts on $X$ by scalar multiplication
on each component of $z\in\C^n$ in the identification \cref{XPie1},
we see from \cref{mu0def} that $T_0$ acts transitively on
the fibres of the moment map $\mu_0:X\to\Delta$,
and freely over the interior $\Delta^0\subset\Delta$.
For any $\nu\in\Delta$, we write
\begin{equation}\label{Tnudef}
T_\nu:=\{g\in T_0~|~g.x=x\text{ for all }x\in\Lambda_\nu\}\,,
\end{equation}
which is a closed subgroup of $T_0$.
By completing a basis of the integral lattice of its
Lie subalgebra $\t_\nu:=\Lie(T_\nu)\subset\t_0$ into an
basis of the integral lattice of $\t_0$, we get a splitting
\begin{equation}\label{splitToTnu}
T_0=T_\nu\times(T_0/T_\nu)\,,
\end{equation}
inducing a freely transitive action of $T_0/T_\nu$ on $\Lambda_\nu$.
Write $\t_0=\t_\nu\oplus(\t_0/\t_\nu)$ for the induced splitting of its
Lie algebra.

Recalling \cref{BS} of a Bohr-Sommerfeld submanifold in $(X,\om)$,
we then have the following basic fundamental result.

\begin{prop}\label{BStoricprop}
The fibre $\Lambda_\nu\subset X$ over $\nu\in\Delta$
of the moment map \cref{mu0def}
for the Hamiltonian action of $T_0\subset U(n)$ on $(X,\om)$
satisfies the Bohr-Sommerfeld condition at level
$p\in\N$ if and only if
\begin{equation}\label{BSprojfla}
\nu\in\Delta\cap\,\Big(\frac{1}{p}\,\Z^{n-1}\Big)\,.
\end{equation}
\end{prop}
\begin{proof}
Let $\nu\in\Delta$, and fix a splitting \cref{splitToTnu}.
For any $p\in\N$, the Kostant formula \cref{KostantLp}
applied to $\mu_0:X\to\Delta$ over the fibre $\Lambda_\nu\subset X$
shows that a section $\zeta^p\in\cinf(\Lambda_\nu,\iota_\nu^*L^p)$ satisfies
$\nabla^{\iota_\nu^*L^p}\zeta^p=0$ if and only if, for all $x\in\Lambda_\nu$
and $\xi\in\t$, we have
\begin{equation}\label{zetafla}
e^{\xi}.\zeta^p(e^{-\xi}x)=e^{2\pi\sqrt{-1}p\<\nu,\xi\>}\zeta^p(x)\,.
\end{equation}
Note that if $\xi\in\t_\nu\subset\t_0$,
so that $e^{\xi}x=x$ by \cref{Tnudef},
equation \cref{zetafla} is automatically verified for any section over
$x\in\Lambda_\nu$ by definition \cref{Lgamdef} of the
line bundle $L:=L_\lambda$. As $e^{\xi}=1$ if and only if $\xi\in\Z^{n-1}$,
a non-vanishing section can satisfy \cref{zetafla}
only if $\<\nu,\xi\>\in\frac{1}{p}\Z$ for all $\xi\in\t_\nu\cap\Z^{n-1}$.

Choose now $x\in\Lambda_\nu$ and
$\eta_x\in L_x$, and define a
section $\eta\in\cinf(\Lambda_b,\iota_b^*L)$
for all $g\in T_0/T_\nu$
in the splitting \cref{splitToTnu} by the formula
\begin{equation}\label{etadef}
\eta(g.x):=g.\eta_x\,,
\end{equation}
which is well-defined since $T_0/T_\nu$ acts freely and transitively
on $\Lambda_\nu$.
We then have $g\eta=\eta$ for all $g\in T_0/T_\nu$ by definition
of the action \cref{action}, so that $L_\xi\,\eta=0$
for all $\xi\in\mathfrak{t}_0/\t_\nu\subset\t_0$ by definition of the Lie
derivative \cref{Lxidef}. We then get
from \cref{zetafla}
that a section non-vanishing $\zeta^p\in\cinf(\Lambda_\nu,\iota_\nu^*L^p)$
satisfies
$\nabla^{\iota_\nu^*L^p}\zeta^p=0$ if and only if there exists $c\in\C^*$
such that for all $x\in\Lambda_\nu$ and $\xi\in\t_0/\t_\nu$, we have
\begin{equation}\label{zetadef}
\zeta^p(e^{\xi}.x)=c\,e^{2\pi\sqrt{-1} p\<\nu,\xi\>} \eta^p(e^{\xi}.x)\,.
\end{equation}
As $e^{\xi}.x=x$ if and only if $\xi\in\Z^n$, we thus
see that \cref{zetadef} defines a non-vanishing section
if and only if $\<\nu,\xi\>\in\frac{1}{p}\Z$ for all
$\xi\in(\t_0/\t_\nu)\cap\Z^{n-1}$.
Hence a non-vanishing section over 
$\Lambda_\nu\subset X$ satisfying \cref{zetafla}
exists
if and only if $\nu\in\Delta$ satisfies \cref{BSprojfla},
which proves the result by the \cref{BS} of the Bohr-Sommerfeld condition
at level $p\in\N$.



\end{proof}

Following for instance \cite[Ex.\,1.25]{Sze14},
recall that for any $p\in\N$ and under the identification
\cref{Xl=CPn},
the space of holomorphic sections $H^0(X,L^p)$ naturally identifies with the
space of
homogeneous polynomials of degree $p$ over $\C^{n}$, and the induced action 
of $U(n)$ on polynomials over $\C^n$ corresponds to the action
\cref{action} of $U(n)$ on $H^0(X,L^p)$. One then readily sees that
for all $\nu=(\nu_2,\dots,\nu_n)\in\N^{n-1}$ such that
$|\nu|:=\sum_{j=2}^n\nu_j\leq p$,
the monomial
\begin{equation}\label{monom}
\begin{split}
z^\nu:\C^n&\longrightarrow\C\\
z&\longmapsto z_1^{p-|\nu|}\prod_{j=2}^n z_j^{\nu_j}
\end{split}
\end{equation}
generates the weight space \cref{wghtspacedef} of weight $\nu\in\Z^{n-1}$
in the decomposition \cref{Tirrepdec} of the space of homogeneous polynomials
with respect to the $T_0\subset U(n)$ action in the sense of
\cref{wghtspacedef},
so that the weight space
$\Gamma_p:=\Gamma_{H^0(X,L^p)}^{T_0}$ defined in \cref{WVdef} is given by
\begin{equation}\label{Wpprojdef}
\Gamma_p=\,\Z^{n-1}\cap p\Delta\,,
\end{equation}
with all weights of multiplicity $1$.
Note that $\frac{1}{p}\Gamma_p$ becomes dense in $\Delta$ as $p\to+\infty$.
Now by the consequence \cref{BWcor2} of the Borel-Weil
\cref{BWth}, the action \cref{action} of $U(n)$
identifies $H^0(X,L^p)$ with
the irreducible representation $V(p\lambda)$ of $U(n)$ with highest
weight $p\lambda=(p,0,\dots, 0)\in\overline{\Gamma}^*_+$
via \cref{GammaUndef},
so that we have a weight decomposition \cref{Tirrepdec}
into $1$-dimensional subspaces
\begin{equation}\label{Vpldec}
V(p\lambda)=\bigoplus_{\nu_p\in\Gamma_p}\C_{\nu_p}\,,
\end{equation}
such that for any $\xi\in\mathfrak{t}_0$ and $v\in\C_{\nu_p}$, we have
$e^\xi.v=e^{2\pi\sqrt{-1}\<\nu_p,\xi\>} v$.
Comparing \cref{Wpprojdef} with \cref{BStoricprop}, we then see that
for any $p\in\N$,
the weight basis of $V(p\lambda)$ is in bijection with the fibres
of the moment map \cref{mu0def} satisfying the Bohr-Sommerfeld condition
at level $p\in\N$.
The following result shows that this is not a coincidence.

\begin{prop}\label{BStoric}
For any $p,\,k\in\N$ and $v_p\in\frac{1}{p}\Gamma_{p-k}\subset\Delta$,
let $f\in\cinf(\Lambda_{v_p},\iota_{v_p}^*L^{-k})$
be invariant by the action of $T_0$, and let
$\zeta^{p}\in\cinf(\Lambda_{v_p},\iota_{v_p}^*L^{p})$
satisfy $\nabla^{\iota_{v_p}^*L^{p}}\zeta^{p}\equiv 0$.
Then the
section $s_{\Lambda_{v_p}}\in H^0(X,L^{p-k})$ defined by
\begin{equation}\label{defLagstateproj}
s_{\Lambda_{v_p}}(x)=\int_{\Lambda_{v_p}} P_{p-k}(x,y).\zeta^p f(y)\,dv_{\Lambda_{v_p}}(y)\,,
\end{equation}
belongs to
$\C_{p v_p}\subset V((p-k)\lambda)$ in the decomposition
\cref{Vpldec}.
Furthermore, if $\zeta^{p}\neq 0$ and $f\neq 0$, then $s_{\Lambda_{v_p}}\neq 0$.
\end{prop}
\begin{proof}
Recall first that the action
\cref{action} of $G$ on $\cinf(X,L^{p-k})$ preserves the orthogonal projection
$P_{p-k}:\cinf(X,L^{p-k})\fl H^0(X,L^{p-k})$ with respect to the $L^2$-product
\cref{L2}, which translates through \cref{Bergdef} to the property that
for any $g\in G$, any $x,\,y\in X$
and $\eta_y\in L^{p-k}_y$, we have
\begin{equation}
g.P_{p-k}(g^{-1}x,y).\eta_y=P_{p-k}(x,gy).(g.\eta_y)\in L^{p-k}_x\,.
\end{equation}
Then by formula \cref{defLagstateproj}, since
$f\in\cinf(\Lambda_{v_p},\iota_{v_p}^*L^{-k})$
is $T_0$-invariant and by formula
\cref{zetafla} for the flat section
$\zeta^p\in\cinf(\Lambda_{v_p},\iota_{v_p}^*L^p)$,
for any $\xi\in\mathfrak{t}$
and $x\in X$, we have
\begin{equation}
\begin{split}
&(e^\xi.s_{\Lambda_{{v_p}}})(x)
=\int_{\Lambda_{{v_p}}}\,e^{\xi}.P_{p-k}(e^{-\xi} x,y).\zeta^pf(y)\,dv_{\Lambda_{{v_p}}}(y)\\
&=e^{2\pi\sqrt{-1}p\<{v_p},\xi\>}\int_{\Lambda_{{v_p}}}\,P_{p-k}(x,e^{\xi}y).\zeta^pf(e^{\xi}y)\,dv_{\Lambda_{{v_p}}}(y)\\
&=e^{2\pi\sqrt{-1}p\<{v_p},\xi\>}\int_{\Lambda_{{v_p}}}\,P_{p-k}(x,y).\zeta^pf(y)\,dv_{\Lambda_{{v_p}}}(y)=e^{2\pi\sqrt{-1}p\<{v_p},\xi\>}s_{\Lambda_{v_p}}(x)\,,
\end{split}
\end{equation}
so that $s_{\Lambda_{{v_p}}}\in\C_{pv_p}$ in the decomposition
\cref{Vpldec}.

Let us now assume that $\zeta^{p}\neq 0$ and $f\neq 0$.
Writing $z^{p{v_p}}\in H^0(X,L^{p-k})$ for the monomial \cref{monom}
associated with $pv_p\in\Gamma_{p-k}$,
one readily sees from \cref{mu0def} that $z^{p{v_p}}$
does not vanish along $\Lambda_{v_p}\subset\CP^{n-1}$.
From formula \cref{zetafla} and picking $x_0\in\Lambda_{v_p}$,
since $z^{p{v_p}}\in\C_{pv_p}$ and
using that
$h^{L}$ is $G$-invariant, for any $\xi\in\mathfrak{t}_0$ we get
\begin{equation}
h^{L^{p-k}}\big(z^{p{v_p}}(e^\xi.x_0),\zeta^pf(e^\xi.x_0)\big)
=h^{L^{p-k}}\big(z^{p{v_p}}(x_0),\zeta^pf(x_0)\big)\,,
\end{equation}
so that $h^{L^{p-k}}(z^{p{v_p}},\zeta^pf)$ is constant along $\Lambda_{v_p}$
since $T_0$ acts transitively.
Using the reproducing property of \cref{proprepgal}, this
implies
\begin{equation}
\begin{split}
\<z^{p{v_p}},s_{\Lambda_{{v_p}}}\>_{L^2}&=\int_{\Lambda_{{v_p}}}
h^{L^{p-k}}(z^{p{v_p}}(x),\zeta^pf(x))\,dv_{\Lambda_{{v_p}}}(x)\\
&=\Vol(\Lambda_{{v_p}},dv_{\Lambda_{v_p}})\,h^{L^{p-k}}(z^{p{v_p}}(x_0),
\zeta^pf(x_0))\neq 0\,,
\end{split}
\end{equation}
which concludes the proof
\end{proof}

For $k=0$, \cref{BStoric} gives an explicit bijection between the
Bohr-Sommerfeld fibres of \cref{mu0def} at level $p\in\N$ as in
\cref{BStoricprop} and a weight basis of $V(p\lambda)$ as in \cref{Vpldec}
via the Borel-Weil \cref{BWth}. In the proof of \cref{toricth} given
in the next section, we will need \cref{BStoric} for $k=n$.

%
%

\subsection{Proof of \cref{toricth}}
\label{prooftoricth}

First recall by \cref{BStoric}, stating that for any $p\in\N$ and any
$v_p\in\frac{1}{p}\Gamma_p$ as in \eqref{Wpprojdef}, the
unit vector
\begin{equation}\label{skpsLkp}
e_{pv_p}:=\frac{s_{\Lambda_{v_p}}}{\|s_{\Lambda_{v_p}}\|}\in H^0(X,L^p)
\end{equation}
belongs to the $1$-dimensional subspace $\C_{pv_p}\subset V(p\lambda)$
in the decomposition \eqref{Vpldec} via the Borel-Weil \cref{BWth}, where
$\{s_{\Lambda_{v_p}}\in H^0(X,L^p)\}_{p\in\N}$ is the isotropic state of
\cref{Lagstate} associated with the sequence of
Bohr-Sommerfeld submanifolds $\{(\Lambda_{v_p},\zeta^p)\}_{p\in\N}$.
On the other hand, since $\mu_0:X\to\Delta$ is a fibration over the interior
of any face of its image $\Delta\subset\R^{n-1}$ by \cref{mu0def},
we readily get that for any sequence
$\{v_p\in\Delta\}_{p\in\N}$ such that $v_p\to v$ as $p\to+\infty$
and belonging to the face of $v\in\Delta$ after some rank, the sequence
$\{\Lambda_{v_p}\subset X\}_{p\in\N}$ converges smoothly towards
$\Lambda_v\subset X$
in the sense of \cref{cvsmoothdef}.
Then the first statement of \cref{toricth} is a straightforward
consequence of \cref{theonorme,theointergal}.

To get the asymptotic expansion \cref{<e1e2>} from the asymptotic
expansion \cref{<u1u2>} for the scalar product of isotropic states,
we need to show that for any $x\in\Lambda_p^{(1)}\cap\Lambda_p^{(2)}$,
the factor $\lambda^{(q)}_p(x):=h^{L^p}(\zeta_1^p(x),\zeta_2^p(x))$
appearing in \cref{<u1u2>} satisfies
\begin{equation}\label{lambnda=ueieta}
\lambda^{(q)}_p(x)=h^{L^p}(\zeta_1^p(x_p),\zeta_2^p(x_p))
\,e^{2\pi\sqrt{-1}p\eta_{p}^{(q)}}
\end{equation}
where $\eta_{p}^{(q)}>0$ is the symplectic area of a disk $D\subset X$
bounded by a path
$\gamma_1\subset\Lambda_p^{(1)}$ joining
any point $x_p\in \Lambda_p^{(1)}\cap\Lambda_p^{(2)}$
to any point $y_p\in \Lambda_p^{(1)}\cap\Lambda_p^{(2)}$,
followed by a path $\gamma_2\subset\Lambda_p^{(2)}$ returning to
$x_p\in \Lambda_p^{(1)}\cap\Lambda_p^{(2)}$, which exists
 under the hypotheses of \cref{toricth} since $\Lambda_p^{(1)}=g\Lambda_{v_p}$ and $\Lambda_p^{(2)}=\Lambda_{w_p}$ are connected and
isotopic to each other.
Let us trivialize $L^p$ over $D$ in such a way that
$\zeta_2^p\equiv 1$ along $\gamma_2$.
Writing $\nabla^{L^p}|_D=d-2\pi\sqrt{-1}p\,\alpha$ in this trivialisation, we have
$\alpha|_{\gamma_2}\equiv 0$ by the flatness condition \cref{nabs=0} for
$\zeta_2^p$, and $\om=d\alpha$ by the prequantization formula \cref{preq}.
Then Stokes theorem gives
\begin{equation}\label{intDomingama}
\int_D\,\om=\int_{\gamma_1}\,\alpha\,.
\end{equation}
On the other hand, by the flatness condition \cref{nabs=0} for $\zeta_1^p$
and the fact that $d\zeta_2^p\equiv 0$ in the chosen trivialization of $L^p$
over $D$, we get
\begin{equation}
\begin{split}
\frac{d}{dt}h^{L^p}(\zeta_1(\gamma_1(t)),\zeta_2(\gamma_1(t)))&=
h^{L^p}\left(\zeta_1(\gamma_1(t)),\nabla^{L^p}\zeta_2(\gamma_1(t))\right)\\
&=2\pi\sqrt{-1}\,p\,\alpha_{\gamma_1(t)}
\,h^{L^p}(\zeta_1(\gamma_1(t)),\zeta_2(\gamma_1(t)))\,,
\end{split}
\end{equation}
so that for all $t\in[0,1]$, we can solve the corresponding ODE to get
\begin{equation}
h^{L^p}(\zeta_1(\gamma_1(t)),\zeta_2(\gamma_1(t)))=
\exp\left(2\pi\sqrt{-1}p\int_0^t\gamma^*\alpha\right)
h^{L^p}(\zeta_1(x_p),\zeta_2(x_p))\,.
\end{equation}
Setting $t=1$, we get \cref{lambnda=ueieta} from \cref{intDomingama}.
Together with \cref{theonorme,theointergal}, this
gives the asymptotic expansion \cref{<e1e2>}.

Let us now consider the assumptions of formula \cref{toricfla},
and recall for instance from \cite[Ex.\,1.27]{Sze14}
that the canonical line bundle of 
$X\simeq\CP^{n-1}$ is given by $K_X=L^{-n}$,
so that when $n\in\N^*$ is even, we get a natural choice of square
root given by $K_{X}^{1/2}=L^{-n/2}$,
with induced Hermitian metric and connection.
In particular, for
all $p\in\N$ we get
\begin{equation}
L^p\otimes K_X^{1/2}=L^{p-\frac{n}{2}}\,.
\end{equation}
The action of $T_0\subset U(n)$ on $L$ induces an action on
$K_{X}^{1/2}$ compatible with its natural action on $K_X$,
and since it acts by isometries on $(X,g^{TX})$,
for any $v_p\in\Delta^0\subset\Delta$,
the Riemannian volume form
$dv_{\Lambda_{v_p}}\in\cinf(\Lambda_{v_p},\det(T\Lambda_{v_p}))$
is invariant with respect to the action of $T_0$.
Letting $dv_{\Lambda_{v_p}}^{1/2}\in\cinf(X,\iota_{v_p}^*K_X^{1/2})$ be
a square root in the identification \cref{isoLKX}
of $\iota_{v_p}^*K_X$ with $\det(T^*_\C\Lambda_{v_p})$,
\cref{BS,BStoric} show that the Lagrangian state
$\{s_{\Lambda_{v_p}}\in H^0(X,L^p\otimes K_X^{1/2})\}_{p\in\N}$
associated with the sequence of Bohr-Sommerfeld Lagrangian submanifolds
$\{(\Lambda_{v_p},dv_{\Lambda_{v_p}}^{1/2},\zeta^p)\}_{p\in\N}$
converging to $\{(\Lambda_v,dv_{\Lambda_v}^{1/2})\}_{p\in\N}$ belong
to $\C_{pv_p}\subset V((p-n/2)\lambda)$ and is non-vanishing
for all $p\in\N$.

Now since $v_p,\,w_p\in\Delta^0\subset\Delta$ for all $p\in\N$
large enough under the assumptions of formula \cref{toricfla},
the $(n-1)$-dimensional torus $T_0\subset U(n)$ acts freely
and transitively on $\Lambda_{v_p}$ and on $\Lambda_{w_p}$, so that,
letting $\{\xi_j\}_{j=1}^{n-1}$ be a
basis of the integral lattice of $\mathfrak{t}_0\subset\u(n)$, we get that
$\{\xi_j^X\}_{j=1}^{n}$
induces a basis of $T_x\Lambda_{w_p}$ at all $x\in\Lambda_{w_p}$ and that
$\{\Ad_g\xi_j^X\}_{i=1}^{n}$
induces a basis of $T_x(g\Lambda_{v_p})$ at all $x\in\,g\Lambda_{w_p}$.
\cref{KKSprop} then gives
\begin{equation}
\om_\alpha(\Ad_g\xi_j^X,\xi_k^X)=\alpha([\Ad_g\xi_j,\xi_k])\,,
\end{equation}
for any $\alpha\in g\Lambda_{v_p}\cap\Lambda_{w_p}\subset\u(n)^*$ and
$1\leq j,\,k\leq n-1$. Furthermore,
by $T_0$-invariance and the fact that $\Vol(T_0)=1$, we get the following
identity of functions over $\Lambda$,
\begin{equation}\label{dvLvp}
dv_{\Lambda_{w_p}}(\xi_1^X,\dots,\xi_n^X)\equiv
\Vol(\Lambda_{w_p},dv_{\Lambda_{w_p}})\,.
\end{equation}
On the other hand, \cref{theonorme} applied to
$\{s_{\Lambda_{w_p}}\in H^0(X,L^p\otimes K_X^{1/2})\}_{p\in\N}$ gives
the following asymptotic expansion as $p\to+\infty$,
\begin{equation}\label{normeproj}
\big\|s_{\Lambda_{w_p}}\big\|_{L^2}=2^{\frac{n}{4}}p^{\frac{n}{4}}\left(
\Vol(\Lambda_{w_p},dv_{\Lambda_{w_p}})^{1/2}
+O(p^{-1})\right)\,,
\end{equation}
Applying \cref{corintergal} to \cref{skpsLkp} with the vectors
$\xi_j:=\Vol(\Lambda_{w_p},dv_{\Lambda_{w_p}})^{-\frac{1}{n-1}}\xi_j^X$
for all $1\leq j\leq n-1$
as well as the analogous formula for
$\{gs_{\Lambda_{v_p}}
=s_{g\Lambda_{v_p}}\in H^0(X,L^p\otimes K_X^{1/2})\}_{p\in\N}$
since $U(n)$ acts by biholomorphic isometries, this gives
formula \cref{toricfla} and concludes the proof.

\section{Quantization of Gelfand-Zetlin systems}
\label{GZquantsec}

In this Section, we fix $n\in\N^*$ and
consider the case of irreducible representations
of $G=U(n)$ with highest weight 
\begin{equation}\label{GammaregUndef}
\lambda\in\overline{\Gamma}^*_+=\{(\lambda_1,\dots,\lambda_n)\in\Z^n
~|~\lambda_1>\lambda_2>\cdots>\lambda_n\}\,,
\end{equation}
in the sense of \cref{highestwghtdef},
by definition \cref{Gamma0def} of the regular dominant weights
under the identifications \cref{Weylchamber} and \cref{GammaUndef}.

\subsection{Gelfand-Zetlin bases}
\label{GZbassec}

Recall from \cref{highestwghtth}
that the unitary irreducible representations
of $U(n)$ are in bijective correspondence
with the discrete set \cref{GammaregUndef},
sending $\lambda\in\overline{\Gamma}_+^*$
to the associated \emph{highest weight representation} $V(\lambda)$.
Recalling also the sequence of inclusions \cref{GLinc},
we have the following result of Weyl, whose statement
can be found in \cite[Prop.\,6.2]{GS83}.

\begin{theorem}\label{Weylprop}
{\cite[Th.\,8.1.1]{GW09}}
For any $\lambda=(\lambda_1,\dots,\lambda_n)\in\Gamma^*_+$,
the associated highest weight representation $V(\lambda)$
admits the following decomposition into irreducible representations
of $U(n-1)\subset U(n)$,
\begin{equation}\label{decompon-1}
V(\lambda)=\bigoplus_{\mu\in \Gamma_1(\lambda)}\,V(\mu)\,,
\end{equation}
where
\begin{equation}
\Gamma_1(\lambda)=\{\mu\in\Z^{n-1}~|~\lambda_1\geq\mu_1\geq\lambda_2\geq\mu_2\geq\cdots
\geq\lambda_{n-1}\geq\mu_{n-1}\geq\lambda_n\}
\end{equation}
and $V(\mu)$ is the irreducible representation of $U(n-1)$ with highest weight
$\mu\in\Z^{n-1}$.
\end{theorem}

For any $k\in\N^*$, let us write $T_k\subset U(k)$
for the maximal torus of the unitary group $U(k)$ given
by diagonal matrices, and $\t_k\subset\u(k)$ for the
associated Lie subalgebra of the Lie algebra
$\u(k):=\text{Lie}(U(k))$. Under the identification
$\t_k^*\simeq\R^k$ given by \cref{t*=Rn}, we get an identification
\begin{equation}
\bigoplus_{k=1}^n\,\t_k^*\simeq\bigoplus_{k=1}^n\R^k\,,
\end{equation}
which we will consider
as the space of upper-triangular $n\times n$
matrices with real coefficients. For each
$1\leq i\leq n$, we write
\begin{equation}\label{pik}
\begin{split}
\pi_i:\bigoplus_{k=1}^n\t_k^*&\longrightarrow\t_i^*\\
\big(\nu_j^{(k)}\big)_{1\leq j\leq k\leq n}&\longmapsto
(\nu_j^{(i)})_{1\leq j\leq i}\,.
\end{split}
\end{equation}
for the associated canonical projection.

The following result follows immediately from \cref{Weylprop} by
decreasing induction on $n\in\N^*$ along the sequence of inclusions
\cref{GLinc}.

\begin{cor}\label{GZdecprop}
For any $\lambda=(\lambda_1,\dots,\lambda_n)\in\Gamma^*_+$,
there exists a unique decomposition
\begin{equation}\label{GZdecfla}
V(\lambda)=\bigoplus_{\nu\in \Gamma(\lambda)}\,\C_\nu\,,
\end{equation}
with
\begin{multline}\label{Gammaldef}
\Gamma(\lambda)
=\{(\nu_j^{(k)})_{1\leq j\leq k\leq n}\in
\bigoplus_{k=1}^n\Z^k~|~\nu_j^{(n)}=\lambda_j\text{ for all }1\leq j\leq n\,,\\
\nu_j^{(k)}\geq\nu_j^{(k-1)}\geq \nu_{j+1}^{(k)}
\textup{ for all }1\leq j\leq k\leq n\}\,,
\end{multline}
where for any $\nu\in\Gamma(\lambda)$, the subspace
$\C_\nu\subset V(\lambda)$ is the $1$-dimensional space characterized by
the fact
that
$\C_\nu\subset V(\pi_k(\nu))$
at each step $1\leq k\leq n-1$ in the decomposition \cref{decompon-1}
by decreasing induction from $n$ to
$k<n$ along the sequence of inclusions \cref{GLinc}.
\end{cor}

For any $\lambda\in\overline{\Gamma}^*_+$,
the decomposition \cref{GZdecfla} is called
the \emph{Gelfand-Zetlin decomposition} of $V(\lambda)$.
Note that \cref{GZdecprop} immediately implies that 
\begin{equation}\label{dimVlambda}
\dim V(\lambda)=\#\Gamma(\lambda)\,,
\end{equation}
as stated in \cite[Prop.\,6.5]{GS83}.
This also leads to the following definition, introducing the main notion used
in this paper.

\begin{defi}\label{GZbasdef}
A \emph{Gelfand-Zetlin basis} of the unitary irreducible representation
$V(\lambda)$ with highest weight
$\lambda\in\overline{\Gamma}^*_+$
is an orthonormal basis $\{e_\nu\}_{\nu\in\Gamma(\lambda)}$
such that $e_\nu\in\C_\nu$ for all
$\nu\in\Gamma(\lambda)$ in the decomposition \cref{GZdecfla}.
\end{defi}

As the Gelfand-Zetlin decomposition \cref{GZdecfla}
of \cref{GZdecprop} is unique, the
Gelfand-Zetlin basis of \cref{GZbasdef}
is also unique up to multiplication of each basis
vector by a unit scalar.

Recall now the sequence of inclusions \cref{GLinc},
where for any $k\in\N$ with $k\leq n$,
the subgroup $U(k)\subset U(n)$ consists of the
two-blocks diagonal matrices with lower-right
$(n-k)\times (n-k)$ block given by the diagonal identity
matrix of $\C^{n-k}$.
For any such $k\in\N$, we write
$\u(k)\subset\u(n)$ for the Lie subalgebra
of $U(k)\subset U(n)$.
%
We then get a natural sequence of inclusions
\begin{equation}\label{gincl}
U(\u(1))\subset U(\u(2))\subset\cdots\subset U(\u(n))\,,
\end{equation}
where for all $1\leq k\leq n$, the subalgebra
$U(\u(k))\subset U(\u(n))$
of the enveloping algebra $U(\u(n))$
is the subalgebra generated by $j(\u(k))\subset U(\u(n))$
via the canonical
inclusion \cref{idef}. The subalgebra
$U(\u(k))\subset U(\u(n))$ is naturally isomorphic to the
enveloping algebra of $\u(k)$, for all $1\leq k\leq n$.


\begin{defi}\label{GZalgdef}
The \emph{Gelfand-Zetlin subalgebra} $\AA_n\subset U(\u(n))$
is the subalgebra generated by the centers
$Z[U(\u(k))]\subset U(\u(n))$ via the inclusion \cref{gincl},
for all $1\leq k\leq n$.
\end{defi}

The following well-known property, which can be found for instance in \cite{DOF91},
is a straightforward consequence of \cref{GZalgdef}.

\begin{prop}\label{GZalgprop}
The Gelfand-Zetlin subalgebra $\AA_n\subset U(\u(n))$ of \cref{GZalgdef}
is commutative.
\end{prop}
\begin{proof}
By \cref{GZalgdef}, it is enough to check the statement
for $z_1,\,z_2\in\AA_n$ satisfying $z_j\in Z[U(\u(k_j))]$ via the inclusion
\cref{gincl}
with $j=1,\,2$ and $k_1\leq k_2$. Now $z_1\in Z[U(\u(k_1))]$
is naturally included in $U(\u(k_2))$ via the sequence of inclusions \cref{gincl},
and as $z_2\in Z[U(\u(k_2))]$ belongs to the center of
$U(\u(k_2))$, we get that $z_1$ and $z_2$
commute. This concludes the proof of the Proposition.
\end{proof}


Recall that for any $1\leq k\leq n$,
the Weyl group of $\u(k)$ coincides with the $k^{\text{th}}$
symmetric group $\mathfrak{S}_k$, acting on
$\t_k^*\simeq\R^k$ by permutation of coordinates under the identification \cref{t*=Rn}.
We consider the cartesian product
\begin{equation}\label{hatW}
\widetilde{W}=\prod_{k=1}^n\,\mathfrak{S}_k\,,
\end{equation}
with induced action on $\bigoplus_{k=1}^n
\t_k^*$.
Recall also formula \cref{rhon} for the half-sum of positive roots $\rho_{\u(k)}\in\t_k^*$, and 
let us set
\begin{equation}\label{hatrho}
\widetilde\rho=(\rho_{\u(1)},\rho_{\u(2)},\dots,\rho_{\u(n)})
\in\bigoplus_{k=1}^n\t_k^*\,.
\end{equation}
Recall the action \cref{Phi} of the enveloping algebra $U(\u(n))$
on any representation $V$ of $U(n)$.
We then have the following result, which will be crucial in the sequel.

\begin{prop}\label{hatgamprop}
The Gelfand-Zetlin subalgebra $\AA_n\subset U(\u(n))$
of \cref{GZalgdef} acts diagonally in the
Gelfand-Zetlin decomposition \cref{GZdecfla}
of \cref{GZdecprop}, and there is an injective
algebra morphism
\begin{equation}\label{hatgamfla}
\widetilde\gamma:\AA_n
\lhook\joinrel
\longrightarrow
\R\big[\bigoplus_{k=1}^n
\t_k^*\big]^{\widetilde{W}}\,,
\end{equation}
characterised by the fact that
for any $\lambda\in\Gamma_+^*$,
the action of $z\in\AA_n$
on the vector $e_\nu\in V(\lambda)$ parametrized by
$\nu\in\Gamma(\lambda)$
in a Gelfand-Zetlin basis of \cref{GZbasdef}
is given by
\begin{equation}\label{evGZ}
\Phi(z).e_\nu=\widetilde\gamma(z)(2\pi\sqrt{-1}(\nu+\widetilde{\rho}))\, e_\nu\,.
\end{equation}
Furthermore, for any $1\leq j\leq n$, the restriction of $\widetilde{\gamma}$
to the subalgebra $Z[U(\u(j))]\subset\AA_n$ via the inclusion \cref{gincl} satisfies
\begin{equation}\label{hatgam=gamk}
\widetilde\gamma\,\big|_{Z[U(\u(j))]}=\pi_j^*\circ\gamma\,.
\end{equation}
\end{prop}
\begin{proof}
Let $\lambda\in\overline{\Gamma}_+^*$, and consider a Gelfand-Zetlin
basis $\{e_\nu\}_{\nu\in\Gamma(\lambda)}$ of the associated
highest weight representation $V(\lambda)$ as in
\cref{GZbasdef}.
Then for any $\nu\in\Gamma(\lambda)$,
any $1\leq k\leq n$ and any
$z\in Z[U(\u(k))]\subset\AA_n$,
\cref{HCth} together with
\cref{GZdecprop} readily implies that
\begin{equation}\label{evGZdem}
\Phi(z).e_\nu=\gamma(z)(2\pi\sqrt{-1}\pi_k(\nu+\widetilde{\rho}))\,e_\nu
\,.
\end{equation}
Since by \cref{GZalgdef}, the Gelfand-Zetlin subalgebra $\AA_n\subset U(\u(n))$
is generated by
$Z[U(\u(k))]\subset U(\u(n))$ for all $1\leq k\leq n$, this readily implies that
$\AA_n$ acts diagonally on any Gelfand-Zetlin basis, with eigenvalues
determined by formula \cref{evGZdem}, and that if a morphism
$\widetilde{\gamma}$ as in \cref{hatgamfla}
satisfying \cref{hatgam=gamk} exists, then
it satisfies \cref{evGZ} and is uniquely determined by this property.

To show the existence of the morphism \cref{hatgamfla},
pick $z\in\AA_n$ and choose a presentation of it
as a linear combination of products of elements of
$Z[U(\u(k))]\subset U(\u(n))$ for all $1\leq k\leq n$.
Through the injective morphism
\begin{equation}
\pi_j^*:\R[\t_j^*]^{\mathfrak{S}_k}
\lhook\joinrel\longrightarrow\R\big[\bigoplus_{k=1}^n
\t_k^*\big]^{\widetilde{W}}\,,
\end{equation}
we define $\widetilde{\gamma}(z)$
to be the unique extension of
\cref{hatgam=gamk} preserving sums and products in this presentation.
By formula \cref{evGZdem},
we see that formula \cref{evGZ} clearly holds
for all $\lambda\in\overline{\Gamma}_+^*$ and $\nu\in\Gamma(\lambda)$,
so that by \cref{Gammaldef}, we see that the values of
$\widetilde{\gamma}(z)$ as a polynomial over
$\R[\bigoplus_{k=1}^n\t_k^*]$ are prescribed on
the integral elements in an open cone of $\bigoplus_{k=1}^n\R^k$,
which implies as explained after \cref{HCth} that formula \cref{evGZ} uniquely determines
$\widetilde{\gamma}(z)$ as a polynomial over
$\R[\bigoplus_{k=1}^n\t_k^*]$.
This shows in particular that $\widetilde{\gamma}(z)$ does not depend
on the chosen presentation of $z\in\AA_n$ and that
it is a morphism of algebras.
This establishes the existence of
the morphism \cref{hatgamfla}.

To see that the morphism \cref{hatgamfla} is injective,
assume by contradiction that there exists a non-zero
$z\in\AA_n$ whose action via \cref{Phi} on every irreducible representation
vanishes identically, and let $P\in\R[\u(n)^*]$ be such
that
\begin{equation}\label{z=symP}
z=\sym{P}\in U(\u(n))
\end{equation}
by the Poincaré-Birkhoff-Witt \cref{PBWprop}.
In particular, there exists $\lambda\in\g^*$ such that $P(\lambda)\neq 0$,
and since any element in $G$ is conjugate to an element in $T\subset G$,
up to replacing $z\in\AA_n$ by $\Ad_g\,z\in\AA_n$ for $g\in G$ such that
$\Ad_g^*\,\lambda\in\t^*$, which also vanishes
on every irreducible representation, one can assume that $\lambda\in\t^*$.
Hence $P\in\R[\u(n)^*]$ does not vanish identically on $\t^*$, and since
$\Gamma^*_+$ consists of the integral elements in the open cone
$\overline{\t^*_+}\subset\t^*$ by \cref{LT} and \cref{Gammadef},
it cannot vanish identically on $\Gamma^*_+$, so that
there exists $\lambda\in\Gamma_+^*$ such that $P(\lambda)\neq 0$.

Now by \cref{Qp=Phi,BWcor2},
for any $\lambda\in\Gamma_+^*$, definition \cref{z=symP} implies
that the operator $Q_p(P)\in\End(H^0(X_\lambda,L_\lambda^p))$
vanishes for all $p\in\N$. On the other hand,
letting $\lambda\in\Gamma_+^*$ such that $P(\lambda)\neq 0$,
since the moment map $\mu:X_\lambda\hookrightarrow\u(n)^*$
of \cref{KKSprop} is just the inclusion,
 the function
$\mu^*P\in\cinf(X_\lambda,\R)$ does not vanish identically.
Now if $Q_p(P)=0$ for all $p\in\N$,
\cref{L=Tp} implies that $T_p(\mu^*P)=O(p^{-1})$
as $p\to +\infty$ for the operator norm,
which then contradicts the norm correspondence \cref{BMS0}
of \cref{BMS}. This shows that \cref{z=symP} cannot vanish
on the irreducible representation $V(p\lambda)$ of highest weight
$p\lambda$ for all $p\in\N$, showing
that the morphism \cref{hatgamfla}
is injective and concluding the proof.

\end{proof}



%

\subsection{Gelfand-Zetlin systems}
\label{GZsystsec}

Recall from \cref{coadsec} that for any $\lambda\in\overline{\Gamma}_+^*$, the
action of $U(n)$ on the associated coadjoint orbit
$X_\lambda\subset\u(n)^*$
is Hamiltonian with moment map $\mu:X_\lambda\hookrightarrow\u(n)^*$
given by the inclusion.
For any $1\leq k\leq n$,
write $p_k:\u(n)^*\rightarrow\u(k)^*$ for
the projection associated with the inclusion $U(k)\subset U(n)$
induced by the sequence of inclusions \cref{GLinc}.
Then the map
\begin{equation}\label{mukdef}
\mu_k:=p_k\circ\mu: X_\lambda\longrightarrow\u(k)^*\,,
\end{equation}
satisfies the Kostant formula \cref{Kostantmufla}, hence
is a moment map
for the action of the subgroup $U(k)\subset U(n)$.
The following definition is adapted from \cite[p.\,119]{GS83}.

\begin{defi}\label{GZsystemdef}
For any $\lambda\in\overline{\Gamma}_+^*$,
the \emph{Gelfand-Zetlin system} over the coadjoint
orbit $X_\lambda$ is the collection of functions
\begin{equation}
H_j^{(k)}=\mu_k^*\,s_j^{(k)}\in\cinf(X_\lambda,\R)
\quad\text{ for all }1\leq j\leq k\leq n-1\,,
\end{equation}
where for all $1\leq k\leq n$ and all $1\leq j\leq k$,
the polynomial $s_j^{(k)}\in\R[\u(k)^*]^{U(k)}$
is given by the $j^{\text{th}}$ coefficient of the
characteristic polynomial
over the set of $k\times k$ Hermitian matrices under the identification
\cref{Herm=un}.
\end{defi}

We then have the following basic result from
\cite[Prop.\,3.1,\,p.\,119]{GS83}, mirroring the analogous
\cref{GZalgprop} on the side of representation theory.

\begin{prop}\label{GZsystemlem}
For any $\lambda\in\overline{\Gamma}^*_+$,
the Gelfand-Zetlin system of \cref{GZsystemdef}
is a family of Poisson-commuting functions inside $\cinf(X_\lambda,\R)$,
\end{prop}
\begin{proof}
Following the proof of \cref{GZalgprop}, let $1\leq k_1,\,k_2\leq n-1$
be such that $k_1\leq k_2$,
and note that for any $1\leq j\leq k_1$ and
$1\leq i\leq k_2$, we have
\begin{equation}\label{bracketH}
\{H_{j}^{(k_1)},H_{i}^{(k_2)}\}
=\mu_{k_2}^*\{p_{k_2,k_1}^*s_j^{(k_1)},s_i^{(k_2)}\}\,,
\end{equation}
where $p_{k_2,k_1}:\u(k_2)^*\to\u(k_1)^*$ denotes the canonical projection
induced by the inclusion $\u(k_1)\subset\u(k_2)$, and where the Poisson bracket on
$\R[\u(k_2)^*]$ is defined by the natural extension of the Lie bracket
of $\u(k_2)$ through the identification \cref{Rg=Sg}. Now as
$s_i^{(k_2)}\in\R[\u(k_2)^*]^{U(k_2)}$ vanishes under the adjoint
action of $\u(k_2)$, the bracket \cref{bracketH} necessarily vanishes,
which concludes the proof.
\end{proof}

Recall now by \cref{Ol=GGl} and \cref{T=Glambda} that if
$\lambda\in\Gamma^*_+$ belongs to the set of dominant
regular weights \cref{Gamma0def},
then we have
\begin{equation}
\dim X_\lambda=\dim U(n)-\dim T=n(n-1)\,.
\end{equation}
By the celebrated \emph{Arnold-Liouville theorem}
\cite[Chap.10,\,\S\,50]{Arn78},
\cref{GZsystemlem} then implies the existence of
\emph{action-angle coordinates} over the open set $U\subset\R^{\frac{n(n-1)}{2}}$
of regular points inside the image of the map
\begin{equation}\label{Hmap}
\begin{split}
H: X_\lambda&\longrightarrow\R^{\frac{n(n-1)}{2}}\\
x&\longrightarrow(H_j^{(k)}(x))_{1\leq j\leq k\leq n-1}\,,
\end{split}
\end{equation}
which means that there exists a diffeomorphism
$\Psi:\Delta^0\to U$
with $\Delta^0\subset\R^{\frac{n(n-1)}{2}}$ open
and a free Hamiltonian action
of the torus $\IT^{\frac{n(n-1)}{2}}$ on
$H^{-1}(U)$ with moment map
\begin{equation}\label{M=PhiH}
M:=\Psi^{-1}\circ H:H^{-1}(U)\longrightarrow\Delta^0\subset\R^{\frac{n(n-1)}{2}}\,.
\end{equation}
%
%
%
To make such action-angle coordinates explicit,
let $1\leq k\leq n$,
and for any $\alpha\in X_\lambda$,
consider the coadjoint orbit
$X_{\mu_k(\alpha)}\subset\u(k)^*$
passing through its image $\mu_k(\alpha)\in\u(k)^*$
under the moment map
\cref{mukdef} for the action of $U(k)\subset U(n)$.
By definition \cref{Wdef}, we know that the intersection
$X_{\mu_k(\alpha)}\cap\t^*_k$ is an orbit of the
Weyl group $W=\mathfrak{S}_k$, and since it acts simply and transitively
on Weyl chambers, the intersection of $X_{\mu_k(\alpha)}$
with the closure
$(\overline{\t_k})_+^*\subset\t_k^*$ of the positive Weyl
chamber \cref{Weylchamber}
%
%
consists of only one point
\begin{equation}\label{Mkalphadef}
M^{(k)}(\alpha)=(M^{(k)}_1(\alpha),M^{(k)}_2(\alpha),\dots,
M^{(k)}_k(\alpha))\in
\t_k^*\simeq\R^k\,.
\end{equation}
Following for instance \cite[\S\,A.1]{FH91},
recall on the other hand that for any $1\leq j \leq k\leq n$,
the $j^{\text{th}}$ coefficient of the characteristic polynomial
$s_j^{(k)}\subset\R[\u(k)^*]^{U(k)}$
used in \cref{GZsystemdef} restricts over $\t_k^*\simeq\R^n$
to the $j^{\text{th}}$ elementary polynomial, which we still write
$s_j^{(k)}\subset\R[\t_k^*]^{\mathfrak{S}_k}$.
We then have the following fundamental result of Guillemin and Sternberg
in \cite{GS83b}, which we present in the form given in \cite[\S\,5]{GS83}.

\begin{theorem}\label{GZsystemprop}
{\cite[\S\,4]{GS83b}}
The map
\begin{equation}\label{GZpol}
\begin{split}
M:X_\lambda &\longrightarrow\R^{\frac{n(n-1)}{2}}\\
\alpha &\mapsto (M^{(k)}_j(\alpha))_{1\leq j\leq k\leq n-1}\,,
\end{split}
\end{equation}
is a continuous map, whose image is the Gelfand-Zetlin polytope
$\Delta_\lambda\subset\R^{\frac{n(n-1)}{2}}$ defined by
formula \eqref{GZpolytopeintro}.
Furthermore, the open subset $M^{-1}(\Delta_\lambda^0)\subset X_\lambda$
over the interior $\Delta_\lambda^0\subset\Delta_\lambda$
is a dense open set over which the map \cref{GZpol} is smooth and regular,
and the diffeomorphism
\begin{equation}\label{Phi-1}
\begin{split}
\Psi:\Delta_\lambda^0& \longrightarrow U\subset\R^{\frac{n(n-1)}{2}}\\
\nu &\longmapsto (s_j^{(k)}(\pi_k(\nu)))_{1\leq j\leq k\leq n-1}\,,
\end{split}
\end{equation}
induces action-angle coordinates for the map
\cref{Hmap} over the open set
$U:=\Psi(\Delta_\lambda^0)$ of its regular values,
with moment map
for the Hamiltonian $\IT^{\frac{n(n-1)}{2}}$-action
over $H^{-1}(U)\subset X_\lambda$ given by \cref{GZpol}.

Finally, the Bohr-Sommerfeld fibres of the map \cref{GZpol} over
the interior
$\Delta_\lambda^0\subset\Delta_\lambda$ are the
fibres over $\Delta_\lambda^0\cap\Z^n$.
\end{theorem}

Note that the map \cref{Phi-1} sends each decreasing sequence
$\nu^{(k)}_1>\nu^{(k)}_2>\cdots>\nu^{(k)}_k$ to the coefficients of the
polynomial $\prod_{j=1}^k(X-\nu_j^{(k)})$, which is clearly 
a diffeomorphism on its image.
%
%
As explained in \cite[p.\,121]{GS83},
under the natural identification
\cref{Herm=un} for all $1\leq k\leq n$,
the Hermitian matrices $\mu_k(\alpha)\in\Herm(\C^k)$
are the $k\times k$ upper-left minors of $\alpha\in\Herm(\C^n)$,
with non-increasing sequence of eigenvalues
$M^{(k)}_1(\alpha)\geq\cdots\geq
M^{(k)}_k(\alpha)$ given by \cref{Mkalphadef}.
We then have the natural bijection
\begin{equation}\label{WGZ}
\begin{split}
\Gamma(\lambda)&\xrightarrow{~\sim~}\Delta_\lambda\cap\Z^{\frac{n(n-1)}{2}}\\
(\nu_j^{(k)})_{1\leq j\leq k\leq n}&\mapsto (\nu_j^{(k)})_{1\leq j\leq k\leq n-1}
\,,
\end{split}
\end{equation}
which is interpreted in \cite[Prop\,6.5]{GS83} as the correspondence between
Bohr-Sommerfeld quantization and holomorphic quantization in the case
of coadjoint orbits of $U(n)$.
By definition \cref{rhobar} of
the $\rho$-shift $\rho\in\Gamma_+^*$,
we then get that for any $p\in\N$, the set
\begin{equation}\label{WpGZ}
\Gamma_p:=\Delta_{p\lambda-\rho}\cap\Z^{\frac{n(n-1)}{2}}\,,
\end{equation}
parametrizes the elements of a Gelfand-Zetlin basis 
of the irreducible representation $V(p\lambda-\rho)$
of $U(n)$ with highest weight $p\lambda-\rho$, and 
is such that $\frac{1}{p}\Gamma_p\subset\R^{\frac{n(n-1)}{2}}$
becomes dense in $\Delta_\lambda$ as $p\to+\infty$, in the sense that
for any open set $U\subset\R^{\frac{n(n-1)}{2}}$ containing $\Delta_\lambda$,
there is $p_0\in\N$ such that $\frac{1}{p}\Gamma_p\subset U$
for all $p\geq p_0$, and for
any $v\in\Delta_\lambda$ there is a sequence
$(v_p\in\frac{1}{p}\Gamma_p)_{p\in\N}$ such that $v_p\to v$ as $p\to+\infty$. 



\subsection{Proof of \cref{mainth}}
\label{proofmainthsec}

Let $\lambda\in\Gamma^*_+$ be a regular weight as in \cref{GammaregUndef},
and consider the associated coadjoint
orbit $X_\lambda\subset\mathfrak{g}^*$, endowed with
its natural symplectic structure $\om$ defined by
\cref{KKS} and the compatible complex structure defined
by \cref{T10O}.  
Recall also the holomorphic Hermitian line bundle
$(L_\lambda,h^{L_\lambda},\nabla^{L_\lambda})$ defined by \cref{Lgamdef},
which prequantizes $(X_\lambda,\om)$ in the sense of \cref{preq}
by \cref{Lpreqprop}. For any $v\in\Delta_\lambda^0\subset\Delta_\lambda$, we write
\begin{equation}
\iota_\nu:\Lambda_\nu:=M^{-1}(\nu)\longhookrightarrow X
\end{equation}
for the fibre of the action-angle coordinates \cref{GZpol} over
$v$.

The proof of \cref{mainth} is based on a result of Charles in \cite{Cha06},
where \cref{corintergal} is established for a more general notion of
Lagrangian states than \cref{Lagstate}. In particular,
this relies on the existence of a metaplectic structure over $(X_\lambda,\om)$,
that is a square root $K_{X_\lambda}^{1/2}$ of the canonical bundle
\cref{KX} of $X_\lambda$. To describe it in our setting,
compare formula \cref{rhon} for the half-sum of positive
roots $\rho_{\u(n)}\in\t^*$ of $\u(n)$ with formula
\eqref{rhobar} for the $\rho$-shift $\rho\in\Z^n$, so that
from \cref{cancoad}, we get
\begin{equation}
K_{X_\lambda}:=L_{-2\rho_\g}=L_{-2\rho}\otimes L_{\theta}\,,
\end{equation}
with $\theta=(n-1,n-1,\dots,n-1)\in\Gamma^*\simeq\Z^n$.
From its definition \cref{Lgamdef}, we see that the holomorphic Hermitian line bundle $(L_{\theta},h^{L_{\theta}})$
is trivial over $X_\lambda$, albeit endowed with a non-trivial
lift of the $U(n)$-action on $X_\lambda$. This means
that for any $n\in\N^*$, the holomorphic Hermitian line bundle
\begin{equation}\label{KX=Ltheta}
(K_{X_\lambda}^{1/2},h^{K_{X_\lambda}^{1/2}})
:=(L_{-\rho},h_{-\rho})
\end{equation}
defines a square root of $K_{X_\lambda}$ over $X_\lambda$.
%
%
By the Borel-Weil \cref{BWth},
for any $p\in\N$, the unitary irreducible representation
$V(p\lambda-\rho)$ of $U(n)$ with highest weight
$p\lambda-\rho$ satisfies
the following identity
of unitary representations of $U(n)$,
\begin{equation}\label{Vplambda+rho}
V(p\lambda-\rho)=(H^0(X_\lambda,L^p_\lambda
\otimes K_{X_\lambda}^{1/2}),
\<\cdot,\cdot\>_{L^2})\,,
\end{equation}
where the lift of the action of $U(n)$ on $K_{X_\lambda}^{1/2}$ is determined by
\cref{KX=Ltheta}.

%

For any $k\in\N$ and any $1\leq j\leq k$, recall that
$s_j^{(k)}\in\R[\u(k)^*]^{U(k)}$ stands for
the $j^{\text{th}}$ coefficient of the characteristic polynomial
over the set of $k\times k$ Hermitian matrices as in \cref{GZsystemdef}
and
$p_k:\u(n)^*\hookrightarrow\u(k)^*$ for the projection induced
by the sequence of inclusion \cref{GLinc}, so that
$p_k^*s_j^{(k)}\in\R[\u(n)^*]$
is a homogeneous polynomial of degree $j\in\N$.
Since on the other hand the
symmetrization map \cref{sym} is equivariant with respect
to the action of $U(n)$, we know that we have
\begin{equation}\label{symps}
\sym(p_k^*s_j^{(k)})\in Z[U(\u(k)]\subset\AA_n
\end{equation}
via the inclusion \cref{gincl}, where
$\AA_n\subset U(\u(n))$ is the Gelfand-Zetlin subalgebra
of \cref{GZalgdef}. Since the coefficients of the characteristic
polynomial generate $\R[\u(k)^*]^{U(k)}$
and via the bijection \cref{symcenter} with $Z[U(\u(k)]$,
elements of the form \cref{symps}
for all $1\leq j\leq k\leq n-1$ generate the
Gelfand-Zetlin subalgebra of \cref{GZalgdef}.
Furthermore, for any $p\in\N$, the action of $\sym(p_k^*s_j^{(k)})$ on the
irreducible representation $V(p\lambda-\rho)$ of $U(n)$
with highest weight $p\lambda-\rho$ is given via
\cref{Qp=Phi} and \cref{Vplambda+rho} by the operator
$Q_p(p_k^*s_k^{(j)})\in\End(V(p\lambda-\rho))$,
and these operators are pairwise commuting for all $1\leq j\leq k\leq n-1$
by \cref{GZalgprop}.

Let now $w\in\Delta_\lambda^0$, consider a sequence
$(w_p\in \frac{1}{p}\Gamma_p)_{p\in\N}$ as in \eqref{WpGZ}
such that $w_p\to w$ as $p\to+\infty$,
and for any $p\in\N$, let
$e_{p w_p}\in V(p\lambda-\rho)$ be the element of
a Gelfand-Zetlin basis parametrized by $p\,w_p\in\Gamma_p$.
By \cref{hatgamprop}, these elements
are common eigenvectors of the operators $Q_p(p_k^*s_j^{(k)})
\in\End(V(p\lambda-\rho))$ for all $1\leq j\leq k\leq n-1$,
which by \cref{corL=Tp} satisfy
\begin{equation}\label{Qpse=He}
\begin{split}
\frac{1}{(2\pi\sqrt{-1} p)^j}Q_p(p_k^*s_j^{(k)})&\,e_{pw_p}\\
&=\frac{1}{(2\pi\sqrt{-1} p)^j}\,
\widetilde\gamma(\sym(p_k^*s_j^{(k)}))(2\pi\sqrt{-1}(p\,w_p+\widetilde{\rho}))\,e_{pw_p}\\
&=(H_{j}^{(k)}|_{\Lambda_{w_{p}}}+O(p^{-1}))\,e_{pw_p}\,,
\end{split}
\end{equation}
by definition \cref{mukdef} of the
application $\mu_k:X\to\u(k)^*$,
where $H_{j}^{(k)}|_{\Lambda_{w_{p}}}\in\mathbb R$ denote the constant values of
the restriction to $\Lambda_{w_{p}}\subset X$ of
the functions $H_j^{(k)}\in\cinf(X,\R)$ for all $1\leq j\leq k \leq n-1$
forming the Gelfand-Zetlin system introduced in \cref{GZsystemdef}.
On the other hand,
\cref{L=Tp}
gives the following asymptotic estimate in the sense of
the operator norm of \cref{Vplambda+rho} as $p\to+\infty$,
\begin{equation}\label{Qps=TpH}
\frac{1}{(2\pi\sqrt{-1} p)^j}Q_p(p_k^*s_j^{(k)})
=T_p(H_j^{(k)})+O(p^{-1})\,.
\end{equation}
By \cref{GZsystemlem}, the functions $H_j^{(k)}\in\cinf(X,\R)$
are pairwise Poisson-commuting
for all $1\leq j\leq k \leq n$.
The result of Charles in \cite[Th.\,5.2]{Cha10b} applied to the family
\cref{Qps=TpH} of pairwise commutating operators for all $1\leq j\leq k \leq n$
then implies that the sequence
$\{e_{p w_p}\in H^0(X_\lambda,L_\lambda\otimes K_{X_\lambda}^{1/2})\}_{p\in\N}$
is a Lagrangian state
associated with the sequence of Lagrangian submanifolds
$\{(\Lambda_{w_p},\zeta^p,f_{w_p})\}_{p\in\N}$
in the sense of \cite[\S\,3]{Cha06}, with
sections
$f_{w_p}\in\cinf(\Lambda_{w_p},\iota_p^*K_{X_\lambda}^{1/2})$
satisfying
\begin{equation}
f^2_{w_p}=dv_{w_p}\in\cinf(\Lambda_{w_p},\det(T^*_\C\Lambda_{w_p}))
\end{equation}
in the identification \cref{isoLKX}, where $dv_{w_p}$
is the volume form of volume $1$ over $\Lambda_{w_p}$ invariant
by the free torus action of $\IT^{\frac{n(n-1)}{2}}$ with moment map
\cref{M=PhiH} induced by the action-angle coordinates of \cref{GZsystemprop}.
It is determined by that fact that
\begin{equation}\label{dvxi=1}
dv_{w_p}\big(\xi_{1}^{X_\lambda},\dots,
\xi_{\frac{n(n-1)}{2}}^{X_\lambda}\big)\equiv 1\,,
\end{equation}
over $\Lambda_{w_p}$, where 
$\{\xi_{j}\}_{j=1}^{\frac{n(n-1)}{2}}$
is the canonical basis of the torus $\IT^{\frac{n(n-1)}{2}}$
acting on $M^{-1}(\Delta^0_\lambda)$ from the definition \cref{M=PhiH}
of action-angle coordinates.

Let now $v\in\Delta_\lambda^0$ and consider a sequence
$(v_p\in \frac{1}{p}\Gamma_p)_{p\in\N}$
such that $v_p\to v$ as $p\to+\infty$.
As any $g\in U(n)$ acts by a biholomorphic isometry,
applying \cite[Th.\,5.2]{Cha10b} as before,
the sequence
$\{ge_{pv_p}\in V(p\lambda-\rho)\}_{p\in\N}$
is also a Lagrangian state
associated with the sequence $\{(g\Lambda_{v_p},g\zeta^p,gf_{v_p})\}_{p\in\N}$
of Bohr-Sommerfeld submanifolds in the sense
of \cite[\S\,3]{Cha06}, so that $gf_{v_p}=g.dv_{v_p}$.
Then if $g\Lambda_v\cap\Lambda_w=\0$,
the rapid vanishing \cref{mainflaexp}
readily follows from \cite[Th.\,5.1]{Cha10b}, as in \cref{theonorme}.

To establish the asymptotics \cref{mainfla},
let us first point out that \cref{Lagstate} of a Lagrangian
state is only a particular case of the general
notion of a Lagrangian state in the sense of
\cite[\S\,3]{Cha06}.
The asymptotics \cref{mainfla} then rather follow from the
formula given in \cite[Th.\,6.1]{Cha10b}
for the Hermitian product of two Lagrangian states with $K=K_{X_\lambda}^{1/2}$,
in strict analogy
with \cref{corintergal}, so that the
asymptotic expansion \cref{mainfla} coincides with the general
form of the formula \cite[Th.\,6.1]{Cha10b}, as computed for instance
in \cite[Th.\,11.1]{Det18},
along with the fact that
$\Vol(\Lambda_{v_p},dv_{v_p})=\Vol(\Lambda_{w_p},dv_{w_p})=1$.
In particular, a description of $\kappa(x)\in\Z/4\Z$
as a Maslov index is given in \cite[\S\,5.4]{Det18}.
As explained
in \cite[\S\,2.3]{Cha03b} and made explicit
in \cite[Th.\,6.1]{Cha10b},
first order asymptotics
of Lagrangian states have a universal behavior,
depending only on the linear symplectic data
at each intersection point of the Lagrangian submanifolds.
Since the local model in flat space for the \cref{Lagstate} of a Lagrangian
state is a Lagrangian state in the sense of
\cite[\S\,3]{Cha06}, and since the first order
asymptotics in \cref{theointergal}, \cref{<u1u2>}, and
\cref{corintergal}, \cref{cor<u1u2>}, are computed using this local model
in \cite{Ioo18b}, one can thus also compute \cref{mainfla} using
\cref{Lagstate} of a Lagrangian state.
We then get formula \cref{mainfla} from formulas \cref{cor<u1u2>},
\cref{lambnda=ueieta} and \cref{dvxi=1}, as well as
the standard fact that 
for all $1\leq j,\,k\leq \frac{n(n-1)}{2}$, we have
\begin{equation}
\om_x(g.\xi_{j}^{X_\lambda},\xi_{k}^{X_\lambda})=\{g_*M_j,M_k\}(x)\,,
\end{equation}
for any $x\in g\Lambda_{v_p}\cap\Lambda_{w_p}$,
since $g\in U(n)$ acts as a Hamiltonian diffeomorphism and since
$\xi_{j}^{X_\lambda}$ is the Hamiltonian vector field over 
$M^{-1}(\Delta_\lambda^0)$ associated with
the $j^{\text{th}}$ component $M_j\in\cinf(X,\R)$ of the map
\cref{GZpol} in $\R^{\frac{n(n-1)}{2}}$
for each
$1\leq j\leq \frac{n(n-1)}{2}$, by its description \cref{M=PhiH} as a moment map
for the $\IT^{\frac{n(n-1)}{2}}$-action on $M^{-1}(\Delta_\lambda^0)$.
This concludes the proof.

\providecommand{\bysame}{\leavevmode\hbox to3em{\hrulefill}\thinspace}
\providecommand{\MR}{\relax\ifhmode\unskip\space\fi MR }
\providecommand{\MRhref}[2]{%
  \href{http://www.ams.org/mathscinet-getitem?mr=#1}{#2}
}
\providecommand{\href}[2]{#2}

\Addresses


\begin{thebibliography}{10}

\bibitem{Akh96}
D.~N. Akhiezer, \emph{Lie group actions in complex analysis}, Aspects of
  Mathematics, vol. E27, Friedr. Vieweg \& Sohn, Braunschweig, 1995.

\bibitem{Arn78}
V.~I. Arnold, \emph{Mathematical methods of classical mechanics}, Graduate
  Texts in Mathematics, vol.~60, Springer-Verlag, New York-Heidelberg, 1978,
  Translated from the Russian by K. Vogtmann and A. Weinstein.

\bibitem{BN23}
B.~Bartlett and N.~Nzaganya, \emph{Coherent loop states and angular momentum},
  Arxiv e-print (2023), arXiv:2306.17293.

\bibitem{BGV04}
N.~Berline, E.~Getzler, and M.~Vergne, \emph{Heat kernels and {D}irac
  operators}, Grundlehren Text Editions, Springer-Verlag, Berlin, 2004,
  Corrected reprint of the 1992 original.

\bibitem{BEP04}
P.~Biran, M.~Entov, and L.~Polterovich, \emph{Calabi quasimorphisms for the
  symplectic ball}, Commun. Contemp. Math. \textbf{6} (2004), no.~5, 793--802.

\bibitem{BHMV95}
C.~Blanchet, N.~Habegger, G.~Masbaum, and P.~Vogel, \emph{Topological quantum
  field theories derived from the {K}auffman bracket}, Topology \textbf{34}
  (1995), no.~4, 883--927.

\bibitem{BMS94}
M.~Bordemann, E.~Meinrenken, and M.~Schlichenmaier, \emph{Toeplitz quantization
  of {K}\"ahler manifolds and {${\rm gl}(N)$}, {$N\to\infty$} limits}, Comm.
  Math. Phys. \textbf{165} (1994), no.~2, 281--296.

\bibitem{BPU98}
D.~Borthwick, T.~Paul, and A.~Uribe, \emph{Semiclassical spectral estimates for
  {T}oeplitz operators}, Ann. Inst. Fourier (Grenoble) \textbf{48} (1998),
  no.~4, 1189--1229.

\bibitem{BdMG81}
L.~Boutet~de Monvel and V.~Guillemin, \emph{The spectral theory of {T}oeplitz
  operators}, Annals of Mathematics Studies, vol.~99, Princeton University
  Press, Princeton, NJ; University of Tokyo Press, Tokyo, 1981.

\bibitem{BdMS75}
L.~Boutet~de Monvel and J.~Sj\"ostrand, \emph{Sur la singularité des noyaux de
  {B}ergman et de {S}zegö}, Journ\'ees: \'Equations aux {D}\'eriv\'ees
  {P}artielles de {R}ennes (1975), Soc. Math. France, Paris, 1976,
  pp.~123--164. Ast\'erisque, No. 34--35.

\bibitem{BD85}
T.~Br\"{o}cker and T.~tom Dieck, \emph{Representations of compact {L}ie
  groups}, Graduate Texts in Mathematics, vol.~98, Springer-Verlag, New York,
  1985.

\bibitem{Cha03b}
L.~Charles, \emph{Quasimodes and {B}ohr-{S}ommerfeld conditions for the
  {T}oeplitz operators}, Comm. Partial Differential Equations \textbf{28}
  (2003), no.~9-10, 1527--1566. \MR{2001172}

\bibitem{Cha06}
\bysame, \emph{Symbolic calculus for {T}oeplitz operators with half-form}, J.
  Symplectic Geom. \textbf{4} (2006), no.~2, 171--198.

\bibitem{Cha10b}
\bysame, \emph{On the quantization of polygon spaces}, Asian J. Math.
  \textbf{14} (2010), no.~1, 109--152.

\bibitem{Cho04}
C.-H. Cho, \emph{Holomorphic discs, spin structures, and {F}loer cohomology of
  the {C}lifford torus}, Int. Math. Res. Not. (2004), no.~35, 1803--1843.

\bibitem{Det18}
R.~Detcherry, \emph{Geometric quantization and semi-classical limits of {TQFT}
  vectors}, Ann. Sci. \'{E}c. Norm. Sup\'{e}r. \textbf{51} (2018), no.~6,
  1599--1630.

\bibitem{Dix96}
J.~Dixmier, \emph{Enveloping algebras}, Graduate Studies in Mathematics,
  vol.~11, American Mathematical Society, Providence, RI, 1996, Revised reprint
  of the 1977 translation.

\bibitem{DOF91}
Y.~A. Drozd, S.~A. Ovsienko, and V.~M. Futorny, \emph{On {G}elfand-{Z}etlin
  modules}, Proceedings of the {W}inter {S}chool on {G}eometry and {P}hysics
  ({S}rn\'{\i}, 1990), no.~26, 1991, pp.~143--147.

\bibitem{FH91}
W.~Fulton and J.~Harris, \emph{Representation theory}, Graduate Texts in
  Mathematics, vol. 129, Springer-Verlag, New York, 1991, A first course,
  Readings in Mathematics.

\bibitem{GZ50}
I.~M. Gelfand and M.~L. Zetlin, \emph{Finite-dimensional representations of the
  group of unimodular matrices}, Doklady Akad. Nauk SSSR (N.S.) \textbf{71}
  (1950), 825--828.

\bibitem{GW09}
R.~Goodman and N.~R. Wallach, \emph{Symmetry, representations, and invariants},
  Graduate Texts in Mathematics, vol. 255, Springer, Dordrecht, 2009.

\bibitem{GS82}
V.~Guillemin and S.~Sternberg, \emph{Geometric quantization and multiplicities
  of group representations}, Invent. Math. \textbf{67} (1982), 515--538.

\bibitem{GS83}
\bysame, \emph{The {G}elfand-{C}etlin system and quantization of the complex
  flag manifolds}, J. Funct. Anal. \textbf{52} (1983), no.~1, 106--128.

\bibitem{GS83b}
\bysame, \emph{On collective complete integrability according to the method of
  {T}himm}, Ergodic Theory Dynam. Systems \textbf{3} (1983), no.~2, 219--230.

\bibitem{H51}
Harish-Chandra, \emph{On some applications of the universal enveloping algebra
  of a semisimple {L}ie algebra}, Trans. Amer. Math. Soc. \textbf{70} (1951),
  28--96.

\bibitem{Hum78}
J.~E. Humphreys, \emph{Introduction to {L}ie algebras and representation
  theory}, Graduate Texts in Mathematics, vol.~9, Springer-Verlag, New
  York-Berlin, 1978, Second printing, revised.

\bibitem{Ioo18b}
L.~Ioos, \emph{Quantization and isotropic submanifolds}, Michigan Math. J.
  \textbf{71} (2022), no.~1, 177--220.

\bibitem{JW92}
L.~C. Jeffrey and J.~Weitsman, \emph{Bohr-{S}ommerfeld orbits in the moduli
  space of flat connections and the {V}erlinde dimension formula}, Comm. Math.
  Phys. \textbf{150} (1992), no.~3, 593--630.

\bibitem{Kna01}
A.~W. Knapp, \emph{Representation theory of semisimple groups}, Princeton
  Landmarks in Mathematics, Princeton University Press, Princeton, NJ, 2001, An
  overview based on examples, Reprint of the 1986 original.

\bibitem{Kos70}
B.~Kostant, \emph{Quantization and unitary representations. {I}.
  {P}requantization}, Lectures in {M}odern {A}nalysis and {A}pplications,
  {III}, Lecture Notes in Math., vol. Vol. 170, Springer, Berlin-New York,
  1970, pp.~87--208.

\bibitem{LM89}
H.~B. Lawson and M.-L. Michelsohn, \emph{Spin geometry}, Princeton Mathematical
  Series, vol.~38, Princeton University Press, Princeton, NJ, 1989.

\bibitem{LY09}
R.~G. Littlejohn and L.~Yu, \emph{Uniform semiclassical approximation for the
  {W}igner 6j-symbol in terms of rotation matrices}, J. Phys. Chem. {A}
  \textbf{113} (2009), no.~52, 14904--149221.

\bibitem{MM07}
X.~Ma and G.~Marinescu, \emph{Holomorphic {M}orse inequalities and {B}ergman
  kernels}, Progress in Mathematics, vol. 254, Birkh\"auser Verlag, Basel,
  2007.

\bibitem{MM08b}
\bysame, \emph{Toeplitz operators on symplectic manifolds}, J. Geom. Anal.
  \textbf{18} (2008), no.~2, 565--611.

\bibitem{NV21}
P.~D. Nelson and A.~Venkatesh, \emph{The orbit method and analysis of
  automorphic forms}, Acta Math. \textbf{226} (2021), no.~1, 1--209.

\bibitem{NNU10}
T.~Nishinou, Y.~Nohara, and K.~Ueda, \emph{Toric degenerations of
  {G}elfand-{C}etlin systems and potential functions}, Adv. Math. \textbf{224}
  (2010), no.~2, 648--706.

\bibitem{PR68}
G.~Ponzano and T.~Regge, \emph{Semiclassical limits of {R}acah coefficients},
  Spectroscopic and group theoretical methods in physics (F.~Bloch, ed.),
  North-Holland Publ. Co., Amsterdam, 1968, pp.~1--58.

\bibitem{RT91}
N.~Reshetikhin and V.~G. Turaev, \emph{Invariants of 3-manifolds via link
  polynomials and quantum groups}, Invent. Math. \textbf{103} (1991), 547--597.

\bibitem{Rob99}
J.~Roberts, \emph{Classical {$6j$}-symbols and the tetrahedron}, Geom. Topol.
  \textbf{3} (1999), 21--66.

\bibitem{Sze14}
G.~Sz\'{e}kelyhidi, \emph{An introduction to extremal {K}\"{a}hler metrics},
  Graduate Studies in Mathematics, vol. 152, American Mathematical Society,
  Providence, RI, 2014.

\bibitem{Tuy87}
G.~M. Tuynman, \emph{Quantization: towards a comparison between methods}, J.
  Math. Phys. \textbf{28} (1987), no.~12, 2829--2840.

\bibitem{Wit89}
E.~Witten, \emph{Quantum field theory and the {J}ones polynomial}, Comm. Math.
  Phys. \textbf{121} (1989), no.~3, 351--399.

\end{thebibliography}
\end{document}